\documentclass[11pt, a4paper]{article}
\usepackage{amsmath, amsthm, amssymb, url, enumerate, subfig}
\usepackage[margin=2cm]{geometry}
\usepackage{tikz, tkz-graph, mathrsfs}

\newcommand{\im}{\operatorname{im}}
\newcommand{\N}{\mathbb{N}}

\usetikzlibrary{arrows}
\newcommand{\vertex}[3]{\node [vertex] (#1) at (#2, #3 * 1.7) {};}
\newcommand{\edge}[2]{\draw (#1) -- (#2);}
\newcommand{\arc}[2]{{\draw[-latex] (#1) edge (#2);}}

\newcommand{\Tran}{\mathrm{Tran}}
\newcommand{\Sing}{\mathrm{Sing}}
\newcommand{\Sym}{\mathrm{Sym}}
\newcommand{\Alt}{\mathrm{Alt}}
\newcommand{\rk}{\mathrm{rk}}
\newcommand{\id}{\mathrm{id}}

\newcommand{\genset}[1]{\langle#1\rangle}
\newcommand{\set}[2]{\{#1:#2\}}

\newcommand{\LL}{\mathscr{L}}
\newcommand{\R}{\mathscr{R}}
\newcommand{\J}{\mathscr{J}}
\newcommand{\K}{\mathscr{K}}
\renewcommand{\H}{\mathscr{H}}
\newcommand{\easty}[2]{(#1 \to #2)}
\newcommand{\nset}{[n]}

\theoremstyle{plain}

\newtheorem{theorem}[equation]{Theorem}
\newtheorem{corollary}[equation]{Corollary}
\newtheorem{proposition}[equation]{Proposition}
\newtheorem{lemma}[equation]{Lemma}
\newtheorem*{claim*}{Claim}

\theoremstyle{definition}

\newcounter{claimcounter}

\numberwithin{claimcounter}{equation}

\newtheorem{claim}[claimcounter]{Claim}

\begin{document}

\title{Structural aspects of semigroups based on digraphs}

\author{James East\footnote{Centre for Research in Mathematics, School of
Computing, Engineering and Mathematics, Western Sydney University,
Locked Bag 1797, Penrith NSW 2751, Australia. Email:
j.east@westernsydney.edu.au}, Maximilien Gadouleau\footnote{School of
Engineering and Computing Sciences, Durham University, South Road, Durham DH1
3LE, UK. Email: m.r.gadouleau@durham.ac.uk}, and James D.
Mitchell\footnote{School of Mathematics and Statistics, University of St
Andrews, St Andrews, Fife KY16 9SS, UK. Email: jdm3@st-andrews.ac.uk}}
\date{\today}
\maketitle

\begin{abstract}
  Given any digraph $D$ without loops or multiple arcs, there is a natural
  construction of a semigroup $\langle D\rangle$ of transformations. To every
  arc $(a,b)$ of $D$ is associated the idempotent transformation $(a\to b)$
  mapping $a$ to $b$ and fixing all vertices other than $a$. The semigroup
  $\langle D\rangle$ is generated by the idempotent transformations $(a\to b)$
  for all arcs  $(a,b)$ of $D$.
  
  In this paper, we consider the question of when there is a transformation in
  $\langle D\rangle$ containing a large cycle, and, for fixed $k\in \mathbb N$,
  we give a linear time algorithm to verify if $\langle D\rangle$ contains a
  transformation with a cycle of length $k$.  We also classify those digraphs
  $D$ such that $\langle D\rangle$ has one of the following properties:
  inverse, completely regular, commutative, simple, 0-simple, a semilattice, a
  rectangular band, congruence-free, is $\mathscr{K}$-trivial or
  $\mathscr{K}$-universal where $\mathscr{K}$ is any of Green's $\mathscr{H}$-,
  $\mathscr{L}$-, $\mathscr{R}$-, or $\mathscr{J}$-relation, and when $\langle
  D\rangle$ has a left, right, or two-sided zero.
\end{abstract}

\section{Introduction}

A \textit{transformation} of degree $n\in \N$ is a function from $\{1, \ldots,
n\}$ to itself.  A \textit{transformation semigroup} is a semigroup consisting
of transformations of equal degree and with the operation of composition of
functions.  For the sake of brevity we will denote $\{1,\ldots,n\}$ by $\nset$.
We define $(a \to b)$ to be the transformation defined by
\begin{equation*} 
v (a \to b) = 
  \begin{cases}
    b & \text{if } v = a\\
    v & \text{otherwise}
  \end{cases} 
\end{equation*}
where $a, b \in \nset$ and $a \ne b$.  A \textit{digraph} is an ordered pair
$(V, A)$, where $V$ is a set whose elements are referred to as
\textit{vertices}, and $A\subseteq V\times V$ is a set of ordered pairs called
\textit{arcs}.  We identify a transformation $(a \to b)$ with an arc $(a, b)$
in a digraph and we refer to $(a \to b)$ as an arc.  If $D$ is a digraph,
we denote by $\langle D \rangle$ the semigroup generated by
the arcs of $D$, and we refer to such a semigroup as \textit{arc-generated}. 

Arc-generated semigroups were first introduced by John Rhodes in the
1960s~\cite[Definition 6.51]{R10}, under the name \textit{semigroups of flows}.
In \cite{R10}, Rhodes was largely concerned with determining the maximal
subgroups of an arc-generated semigroup, and he conjectured that every such
subgroup was a direct product of cyclic, alternating, and symmetric groups.
This conjecture was recently proved in a remarkable paper~\cite{HNP17} by
Horv{\'a}th, Nehaniv, and Podoski. 

Many famous examples of semigroups are arc-generated.  Perhaps the best known
example is the semigroup $\Sing_n$ of all non-invertible, or singular,
transformations on $\nset$, which was shown to be arc-generated by  J.~M.~Howie
in~\cite{H66}.  Other examples include the semigroup of singular
order-preserving transformations~\cite{A62}, and the so-called
\textit{Catalan semigroup}~\cite{H93, S96}, which are generated by the arcs of 
the digraphs $\set{(i, i + 1), (i + 1, i)}{i\in\{1, \ldots, n -1\}}$ 
and $\set{(i, i + 1)}{i\in\{1, \ldots, n -1\}}$, respectively;
these digraphs can be seen in Figure~\ref{fig-catalan}.

\begin{figure}
\begin{center}
\begin{tikzpicture}[scale=1.2,vertex/.style={circle, 
  draw, 
  inner sep=0.08cm}]
	\node [vertex] (q1) at  (0,1) {1};
	\node [vertex] (q2) at  (1,1) {2};
	\node [vertex] (q3) at  (2,1) {3};
	\node [vertex] (qn) at  (4,1) {$n$};
         \draw[-{latex}] (q1) to [bend left=20] (q2);	
         \draw[-{latex}] (q2) to [bend left=20] (q1);	
         \draw[-{latex}] (q2) to [bend left=20] (q3);	
         \draw[-{latex}] (q3) to [bend left=20] (q2);	
         \draw[dotted] (q3)--(qn);
	\node () at (6,1) {and};
\begin{scope}[shift={(8,0)}]
	\node [vertex] (q1) at  (0,1) {1};
	\node [vertex] (q2) at  (1,1) {2};
	\node [vertex] (q3) at  (2,1) {3};
	\node [vertex] (qn) at  (4,1) {$n$};
         \draw[-{latex}] (q1) to (q2);	
         \draw[-{latex}] (q2) to (q3);	
         \draw[dotted] (q3)--(qn);
\end{scope}
\end{tikzpicture}
  \caption{The digraphs $D$ where $\genset{D}$ is the semigroup of singular 
  order-preserving transformations (left) or the Catalan semigroup (right).}
  \label{fig-catalan}
\end{center}
\end{figure}
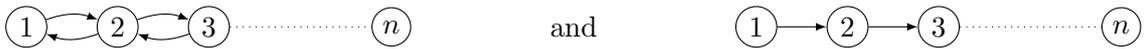

In~\cite{H78}, Howie showed that $\Sing_n$ is generated by $\frac{1}{2}n(n-1)$,
but no fewer, arcs.  In \cite{GH87} it was shown that $\frac{1}{2}n(n-1)$ is
the minimum size of any generating set for $\Sing_n$ whether it consists of
arcs or not.  It was shown in~\cite{H78} that $\Sing_n$ is generated by the
arcs of a digraph $D$ if and only if $D$ is strongly connected and $D$ contains
a tournament.  As a corollary, the minimal-size idempotent generating sets of
$\Sing_n$ are in one-one correspondence with the strongly connected tournaments
on~$n$ vertices; these were enumerated by Wright \cite{W70}.  In \cite{DE16} it
was shown that a digraph $D$ is strongly connected and contains a tournament if
and only if it contains a strongly connected tournament. Hence every
idempotent generating set for $\Sing_n$ contains one of minimum size, something
that is not true for generating sets of $\Sing_n$, in general. 

Several authors have classified those digraphs $D$ such that $\genset{D}$ has a
specific semigroup property. For instance, in \cite{YY06} those digraphs $D$
such that $\genset{D}$ is regular are classified; and in \cite{Cameron2016aa}
those $D$ where $\genset{D}$ is a band are classified.  In \cite{YY06, YY09},
necessary and sufficient conditions on digraphs $D$ and $D'$ are given so that
$\genset{D} = \genset{D'}$ or $\genset{D}\cong\genset{D'}$, respectively.  In
this paper, we continue in this direction, by classifying those digraphs $D$
for which the semigroup $\genset{D}$ has one of a variety of properties.

The paper is organised as follows. In Section~\ref{sec:preliminaries} we review
some relevant terminology and basic results about digraphs and semigroups. In
Section~\ref{sec:cyclic} we investigate the presence of transformations with
long cycles in arc-generated semigroups and classify those arc-generated
semigroups that are $\H$-trivial. It is possible that
Proposition~\ref{prop:L+R} and the converse of Proposition~\ref{prop:l=2} in
Section~\ref{sec:cyclic} can be proved using Theorem~1(5)(b) and Lemma 15
from~\cite{HNP17}. Our proofs were produced independently of the results
in~\cite{HNP17}, and are relatively concise and self-contained, and so we have
included the proofs for the sake of completeness.  Arc-generated semigroups
that are $\LL$-, $\R$- or $\J$-trivial are classified in
Section~\ref{sec:Green}.  Further classes of arc-generated
semigroups---including  bands, completely regular semigroups, inverse
semigroups, semilattices and commutative semigroups---are classified in
Section~\ref{sec:properties}. Finally, properties related to left and right
zeros are classified in Section~\ref{sec:zeros}, which among other things,
allows us to classify those arc-generated semigroups that are rectangular
bands, simple, $0$-simple, or congruence-free.

Many of the results in Sections~\ref{sec:Green},~\ref{sec:properties},
and~\ref{sec:zeros} were suggested by initial computational experiments
conducted using the Semigroups package~\cite{JDM17} for GAP~\cite{GAP4}. 


\section{Preliminaries} \label{sec:preliminaries}
\subsection{Digraphs}
In this subsection, we review some terminology and basic results on digraphs.
We refer the reader to \cite{BG09a} for an authoritative account of digraphs.

Unless otherwise stated, the vertex set of
a digraph will be $\nset$ for some $n\in \N$. 

The \textit{in-degree} of a vertex $v$ in a digraph $D$ is the number of arcs
of the form $(u, v)$ in $D$; similarly, the \textit{out-degree} of $v$ is the
number of arcs of the form $(v, u)$ in $D$.  A vertex $v$ in a digraph $D$ is
called a \textit{sink} if the out-degree of $v$ is $0$. A vertex is
\textit{isolated} if it has no incoming or outgoing arcs. 

If $D = (V, A)$ is a digraph, and $U$ is subset of the vertices $V$ of $D$,
then the \textit{subdigraph of $D$ induced by $U$} is the digraph with vertices
$U$ and arcs $A \cap (U\times U)$. In general, a \textit{subdigraph} of $D =
(V,A)$ is any digraph $D' = (V', A')$ with $V' \subseteq V$ and $A' \subseteq
A\cap (V' \times V')$.

If $D=(V,A)$ is a digraph, and $\varepsilon$ is an equivalence relation on $V$,
then the \textit{quotient digraph} $D/\varepsilon$ is defined as follows.  The
vertex set is the set of all $\varepsilon$-classes of $V$, and if $W,U$ are
$\varepsilon$-classes, then $D/\varepsilon$ has the arc $(W,U)$ if and only if
$W\not=U$ and $D$ has an arc $(w,u)$ for some $w\in W$ and $u\in U$.

A \textit{walk} in a digraph is a finite sequence $(v_0, v_1, \ldots,
v_r)$, $r\geq 1$, of vertices such that $(v_i, v_{i + 1})$ is an arc for all
$i\in \{0, \ldots, r - 1\}$; the \textit{length} of this walk is $r$. A
\textit{path} is a walk where all vertices are distinct. A \textit{cycle} in a
digraph is a walk where $v_0 = v_r$ and all other vertices are distinct.
A digraph is called \textit{acyclic} if it has no cycles.

A \textit{graph} $G$ is defined to be a digraph where $(u, v)$ is an arc if and
only if $(v, u)$ is an arc in $G$. We refer to the pair of arcs above as the
\textit{edge} $\{u, v\}$. Vertices $u$ and $v$ of a graph $G$ are
\textit{adjacent} if $\{u, v\}$ is an edge of $G$. 

An induced subdigraph of a graph, is also a graph, which we refer to as an
\textit{induced subgraph}.  A \textit{spanning subgraph} (as opposed to an
induced subgraph) of a graph $G = (V, A)$ is any graph $H = (V, B)$ where $B
\subseteq A$. 

The \textit{degree} of a vertex in a graph is its in-degree, which equals
its out-degree.  If $u$ and $v$ are vertices of a graph $G$, then the
\textit{distance} from $u$ to $v$ is the length of a shortest path from $u$ to
$v$, if such a path exists.

If $G$ and $H$ are graphs, then $H$ is a \textit{minor} of $G$ if $H$ can be
obtained by successively deleting vertices, deleting edges, or contracting
edges of $G$. 

If $v$ is a vertex of a digraph $D$, then the \textit{strong component} of $v$
is the induced subdigraph of $D$ with vertices $u$ such that there is a path
from $u$ to $v$ and from $v$ to $u$. The \textit{underlying graph} of a digraph
$D$ is the graph with an edge $\{u,v\}$ for each arc $(u,v)$ of $D$. The
\textit{component} of $v$ is the induced subdigraph of $D$ with vertices $u$
such that there is a path from $u$ to $v$ in the underlying graph of $D$.
Every digraph is partitioned by its strong components and its components, and
the quotient of a digraph by its strong components is acyclic.  A
\textit{terminal} component of a digraph $D$ is a strong component $C$ such
that $(u, v)$ is not an arc in $D$ for all $u \in C$, $v \notin C$.
Alternatively, $C$ is terminal if it is a sink in the quotient of a digraph by
its strong components.  A strong component or a digraph is \textit{trivial} if
it only has one vertex. 

A graph $G$ is \textit{separable} if it can be decomposed into two connected
induced subgraphs $G_1$ and $G_2$ with exactly one vertex in common; a graph is
\textit{non-separable} if $G$ admits no such decomposition.  A \textit{block}
of a graph is an induced subgraph that is non-separable and is maximal with
respect to this property.

A graph $G$ is \textit{bipartite} if it can be decomposed into two empty
subgraphs $G_1$ and $G_2$. In other words, in a bipartite graph $G$ every edge
of connects a vertex from $G_1$ with a vertex from $G_2$.  A graph is
\textit{odd bipartite} if it is bipartite and it has an odd number of vertices.

We denote by $K_n$ the \textit{complete graph} with vertices $\nset$
and edges $\{u,v\}$ for all distinct $u,v \in \nset$; by $K_{k, 1}$
the \textit{star graph} with vertices $\{1, \ldots, k + 1\}$ and edges
$\{i, k + 1\}$ for all $i \in \{1, \ldots, k\}$. We denote by $P_n$ the
\textit{path graph}, or simply \textit{path} if there is no ambiguity, with vertices
$\nset$ and edges $\{i, i + 1\}$ for all $i \in \{1, \ldots, n-
1\}$. We denote by $C_n$ the \textit{cycle graph} with vertices $\{1, \ldots,
n\}$ and edges $\{1, n\}$ and $\{i, i + 1\}$ for all $i \in \{1, \ldots, n-
1\}$.


\subsection{Semigroups and monoids}
In this subsection, we review some terminology about semigroups.
We refer the reader to \cite{H95} and \cite{GM09} for further background
material about semigroups.

A \textit{semigroup} is a set with an associative binary operation.  A
\textit{monoid} is a semigroup $S$ with an \textit{identity}: i.e.\  an element
$e\in S$ such that $es = se = s$ for all $s \in S$. 
If $S$ is a semigroup, then $s\in S$ is an \textit{idempotent} if $s ^ 2 = s$.
If $S$ is a semigroup, then we denote by $S ^ 1$ the monoid obtained by
adjoining an identity $1_S\not \in S$ to $S$. 
A semigroup $S$ is \textit{regular} if for all $x\in S$ there exists $y\in S$
such that $xyx = x$. 
A \textit{subsemigroup} of a semigroup $S$ is a subset $T$ of $S$ that is
also a semigroup with the same operation as $S$; denoted $T\leq S$.

A \textit{congruence} on a semigroup $S$ is an equivalence relation
$\varepsilon$ on $S$ for which $(a,b)\in\varepsilon$ and $(c,d)\in\varepsilon$
imply $(ac,bd)\in\varepsilon$ for all $a,b,c,d\in S$.  A semigroup $S$ is
\textit{congruence-free} if the only congruences on $S$ are the universal and
trivial relations.

Let $S$ be a semigroup and let $x,y\in S$ be arbitrary. We say that $x$ and $y$
are \textit{$\LL$-related} if the principal left ideals generated by $x$ and $y$
in $S$ are equal; in other words, if $S^1x = \set{sx}{s\in S ^ 1}=S^1y =
\set{sy}{s\in S ^ 1}$.  We write $x\LL y$ to denote that $x$ and $y$ are
$\LL$-related.

Green's $\R$-relation is defined dually to Green's $\LL$-relation; Green's
$\H$-relation is the meet, in the lattice of equivalence relations on $S$, of
$\LL$ and $\R$. Green's $\J$-relation is defined so that $x,y \in S$ are
\textit{$\J$-related} if $S ^ 1 x S ^ 1 = S^1 y S ^ 1$.  We will refer to the
equivalence classes as $\mathscr{K}$-classes where $\mathscr{K}$ is any of
$\R$, $\LL$, $\H$, or $\J$.  We write $x\K y$ to indicate $(x, y) \in \K$, where
$\mathscr{K}$ is any of $\R$, $\LL$, $\H$, or $\J$.

We denote by $\Tran_n$ the monoid consisting of all of the transformations of
degree $n$ where $n\in \N$; called the \textit{full transformation monoid}.
This monoid plays the same role in semigroup theory as the symmetric group does
in group theory, in that every finite semigroup is isomorphic to a subsemigroup
of some $\Tran_n$. Green's relations on $\Tran_n$ can be described in terms of
the following natural parameters associated to transformations.  The
\textit{image} of a transformation $\alpha\in \Tran_n$ is the set
$$\im(\alpha)=\set{x\alpha}{x\in\{1, \ldots, n\}};$$
the \textit{kernel} of $\alpha$ is the equivalence relation
$$\ker(\alpha)=\set{(x,y)\in\{1, \ldots, n\}\times\{1, \ldots,
n\}}{x\alpha=y\alpha};$$
and the \textit{rank} of $\alpha$ is
$$\rk(\alpha)=|\im(\alpha)|.$$

It is well-known that two elements of $\Tran_n$ are $\R$-, $\LL$- or $\J$-
related if and only if they have the same kernel, image or rank, respectively;
see \cite[Exercise 2.6.16]{H95}.

A semigroup is \textit{aperiodic} if all of its subgroups are trivial. A
semigroup is \textit{$\mathscr{K}$-trivial} for $\mathscr{K} \in \{\R, \LL, \H,
\J\}$, if $x\mathscr{K} y$ implies $x = y$. 


\subsection{Arc-generated semigroups}

We now characterise some basic semigroup theoretic properties of
$\genset{D}$ in terms of digraph theoretic properties of $D$.

Suppose that $D$ is a digraph with vertex set $V$, that $v\in V$ is an isolated
vertex, and that $D'$ is the subdigraph of $D$ induced by $V\setminus \{v\}$.
Then it is clear that the arc-generated semigroups $\genset{D}$ and
$\genset{D'}$ are isomorphic. So, we may assume without loss of generality,
where appropriate and if it is convenient, that a digraph $D$ has no isolated
vertices.

The following proposition will allow us to only consider connected digraphs in
some cases; its proof is trivial and is omitted.  

\begin{proposition}\label{prop-direct-product}
  Let $D$ be a digraph with components $D_1, D_2, \ldots, D_k$ and no
  isolated vertices. Then $\genset{D}$ is isomorphic to
  $\genset{D_1} ^ 1\times\cdots\times \genset{D_k} ^ 1\setminus
  \{(1_{\genset{D_1}}, \ldots, 1_{\genset{D_k}})\}$.  \qed
\end{proposition}

The next result is also trivial.

\begin{proposition}\label{prop-trivial}
  Let $D$ be a digraph. Then the following are equivalent: 
  \begin{enumerate}[\rm (i)]

    \item
      $\genset{D}$ is trivial;

    \item
      $\genset{D}$ is a group;

    \item
      $\genset{D}$ has a unique $\H$-class;

    \item
      $D$ has only one arc. \qed
  \end{enumerate}
\end{proposition}

The semigroup $\genset{D}$ can contain arcs that are not present in $D$. It was
shown in \cite[Lemma~2.3]{YY06} that the set of arcs in $\genset{D}$ is 
$$\{(a \to b) : (a \to b) \in D \text{ or } (b \to a) \text{ belongs to a
cycle of } D \}.$$
The \textit{closure} of $D$, denoted $\bar{D}$, is the digraph on $\nset$ with the set of arcs as above; it is clear that $\genset{\bar{D}} =
\genset{D}$. By construction, $\bar{D} = D$ if and only if every strong
component $C$ of $D$ is a graph. We say that $D$ is \textit{closed} if $\bar{D}
= D$. 


\section{Cyclic properties} \label{sec:cyclic}

A \textit{cycle} of length $k$ in $\alpha\in \Tran_n$ is a sequence of distinct
points $a_0, a_1, \dots, a_{k-1}\in\nset$ such that $a_i \alpha = a_{i+1}$ for
all $i$, where the indices are computed modulo $k$.  This section is concerned
with cycles of transformations in an arc-generated semigroup.  In particular,
we are interested in the presence of \textit{long} cycles.  As mentioned in the
introduction, the results in this section are related to those in~\cite{HNP17},
where the authors describe the structure and actions of the maximal subgroups
of any arc-generated semigroup $\genset{D}$ in terms of properties of $D$.
Proposition~\ref{prop:L+R} and the converse of Proposition~\ref{prop:l=2} could
be proved using Theorem~1(5)(b) and Lemma 15 from~\cite{HNP17}. However,
determining the properties of the specific digraphs in
Propositions~\ref{prop:L+R} and~\ref{prop:l=2} required to apply the results
in~\cite{HNP17} is non-trivial, and requires the notation and terminology used
in~\cite{HNP17}.  Since the proofs presented in this section are
self-contained, and relatively concise, and were found independently
of~\cite{HNP17}, we have opted not to use the results of~\cite{HNP17}.

One application of the results in this section is a classification of the
digraphs $D$ such that $\genset{D}$ is $\H$-trivial.  In particular, we shall
prove the following result.  
\begin{proposition} \label{prop:H-trivial}
  Let $D$ be a digraph. Then $\genset{D}$ is $\H$-trivial if and only if all
  the strong components of $D$ are paths.
\end{proposition}

\subsection{Preliminary results}
The length of a longest cycle of $\alpha$ is denoted as $l(\alpha)$ and for a
digraph $D$ we write
$$l(D) = \max \{ l(\alpha) : \alpha \in \genset{D} \}.$$

\begin{lemma} \label{lem:H-trivial}
  Let $D$ be a digraph.  Then the following are equivalent:
  \begin{enumerate}[\rm (i)]
    \item $l(D) = 1$;
    \item $\genset{D}$ is aperiodic;
    \item $\genset{D}$ is $\H$-trivial.
  \end{enumerate}
\end{lemma}
\begin{proof}
  Conditions (ii) and (iii) are equivalent for any finite semigroup;
  see~\cite[Proposition 4.2]{P86}. 

  (i) $\Rightarrow$ (ii). We prove the contrapositive.  Suppose
  $\alpha\in\genset{D}$ belongs to a non-trivial subgroup and that $\alpha$ is
  not an idempotent.  Then the restriction of $\alpha$ to $\im(\alpha)$ is a
  non-trivial permutation, so $l(\alpha)\geq2$.

  (ii) $\Rightarrow$ (i). Again, we prove the contrapositive.  Suppose
  $\alpha\in\genset{D}$ has a cycle of length $k\geq2$, say $a_0, a_1, \dots,
  a_{k-1}$.  Choose $r\geq1$ such that $\alpha^r$ is an idempotent, and let $H$ be
  the $\H$-class of $\alpha^r$.  Then $H$ is a group, and $H=\set{\alpha^s}{s\geq
  r}$.  But $\alpha^r$ and $\alpha^{r+1}$ are distinct elements of $H$, since
  $a_0\alpha^r=a_0\not=a_1=a_0\alpha^{r+1}$.
\end{proof}

Since $\genset{D}=\genset{\bar D}$ for any digraph $D$, we clearly have
$l(D)=l(\bar D)$.  Thus, when studying $l(D)$, we can assume without loss of
generality that $D$ is closed. If $\alpha \in \genset{D}$, then any cycle of
$\alpha$ belongs entirely to a strong component of $D$. Therefore, if $D$ has
strong components $S_1, \dots, S_r$, then
\begin{equation}\label{eq:lS1...r}
  l(D) = \max \{ l(S_1), \dots, l(S_r)  \}.
\end{equation}
Thus, in this section we may assume without loss of generality, if it is
convenient, that $D$ is a connected graph. 

\begin{lemma} \label{lem:l(Cn)}
  For the cycle graph $C_n$, $n \ge 3$, we have $l(C_n) = n-1$.
\end{lemma}
\begin{proof}
  Clearly, for any digraph $G$ on $n$ vertices, $l(G) \le n-1$. Conversely, 
  $$
  (n-1\to n)(n-2\to n-1)\cdots(1\to2)(n\to1)\in\genset{C_n}
  $$
  has the cycle $1,2,\ldots,n-1$.
\end{proof}

If $G'$ is a subgraph of $G$, then $\genset{G'}\leq\genset{G}$, and so $l(G')
\le l(G)$, which will allow us to isolate subgraphs of $G$ in order to obtain
lower bounds on $l(G)$. In the next lemma, we extend this result to graph minors.

\begin{lemma} \label{cor:l_minor}
  If $G$ is a graph and $H$ is a minor of $G$, then $l(H) \le l(G)$.
\end{lemma}
\begin{proof}
  For any $k \le n$ and $S \le \Tran_k$ and any $T \le \Tran_n$, we write $S
  \preceq T$ if, relabelling the vertices of $\{1, \ldots, k\}$ if necessary,
  for any $\alpha \in S$, there exists $\beta \in T$ such that $v \beta = v
  \alpha$ for all $v \in \{1, \ldots, k\}$.  It is clear that if $\genset{G}
  \preceq \genset{H}$, then $l(G) \leq l(H)$.  Hence it suffices to show that
  $\genset{H} \preceq \genset{G}$.  

  This clearly holds if $H$ is obtained from $G$ from deleting an edge or a
  vertex. Suppose that $H$ is obtained from $G$ by contracting the edge $\{n-1,
  n\}$. Let $B$ be the set of vertices that are adjacent to~$n$ but not to
  $n-1$ in $G$. 

  Let $\alpha \in \genset{H}$ be arbitrary. Then there exist arcs $\beta_1,
  \ldots, \beta_k \in H$ such that $\alpha = \beta_1 \beta_2 \cdots \beta_k$.
  If $\beta_i = (b \to n-1)$ for some $b\in B$, then we replace $\beta_i$ in
  the product for $\alpha$ by $(b \to n)(n \to n-1)$. Similarly, we replace any
  arc $(n-1 \to b)$, $b\in B$, by $(n-1 \to n)(n \to b)$.  If $\beta \in
  \genset{G}$ denotes this modified product, then $v \beta = v \alpha$ for all
  $v \in \{1, \ldots, n -1\}$, and so $\genset{H}\preceq \genset{G}$. 
\end{proof}

\begin{lemma} \label{lem:degree}
  Let $G$ be a graph. Then the following hold:
  \begin{enumerate}[\rm (i)]
    \item
      if $G$ has a vertex of degree $k$, then $l(G) \ge k - 1$;
    \item
      if $G$ contains a subgraph that is a tree with $k$ leaves, then 
      $l(G) \ge k - 1$;
    \item
      if $G$ is connected and $t$ is the number of vertices of degree not
      equal to $2$, then 
      $$l(G) \ge \frac{1}{4}(t-2) + 1.$$
  \end{enumerate}
\end{lemma}
\begin{proof}
  (i).
  For distinct vertices $u$ and $v$ of the star graph $K_{k,1}$ such that $u,
  v\not = k + 1$, we write $(u \rightsquigarrow v) = (u \to k+1) (k+1 \to v)$.
  Then 
  $$
  (k-1\rightsquigarrow k)(k-2\rightsquigarrow
  k-1)\cdots(1\rightsquigarrow2)(k\rightsquigarrow1)\in\genset{K_{k,1}}
  $$
  has the cycle $1,2,\ldots,k-1$.  The result now follows from
  Lemma~\ref{cor:l_minor}.

 (ii).
  If $T$ is a tree with $k$ leaves, then $K_{k,1}$ is a minor of $T$, so the
  result follows from Corollary~\ref{cor:l_minor} and part (i).

  (iii).
  By \cite{BK12}, any graph with $t$ vertices of degree not equal to $2$
  contains a spanning tree with at least $\frac{1}{4}(t-2) + 2$ leaves.
\end{proof}

The next result concerns connected graphs that can be decomposed into two
connected induced subgraphs with a path connecting them.  With this in mind, we
introduce a construction based on paths. Let $L$ and $R$ be two connected
graphs on $m$ and $s$ vertices, respectively, where $m \le s$, and let~$P$ be a
path with $q$ vertices.  Let $L \oplus_q R$ denote the graph obtained by adding
an edge between an endpoint of $P$ and a vertex of $L$ of degree not equal to
1, and an edge between the other endpoint of $P$ to a vertex of $R$ of degree
not equal to 1.  Even though this definition depends on the choice of
attachment vertices, we will omit them in the notation, for our purpose is to
derive results that do not depend on them, apart from the fact that they do not
have degree 1 in $L$ and $R$. We remark that $m, s \ne 2$, since the only
connected graph on two vertices is $K_2$, whose vertices both have degree 1.
However, it is possible to have $m=1$ or $s=1$.

The vertices of $R$ are denoted as $r_1, \dots, r_s$, where they are sorted in
weakly increasing order of distance to the path $P$. In particular, $r_1$ is
attached to $P$, and $r_2$ and $r_3$ are neighbours of $r_1$ if $s\not=1$. A
similar notation is used for $L$; in particular $l_1$ is attached to $P$. Write
$L^* = L \setminus \{ l_1 \}$, $R^* = R \setminus \{ r_1 \}$, $P^* = P \cup
\{l_1, r_1\}$, and order the elements of the path as $p_1, \dots, p_q$ so that
$p_1$ is adjacent to $l_1$, and $p_q$ is adjacent to $r_1$; finally, we also
write $l_1 = p_0$ and $r_1 = p_{q+1}$.  For instance, the graph $K_{3,1}
\oplus_4 C_4$ is illustrated in Figure~\ref{fig:oplus}.

\begin{figure}[ht]
\centering
\begin{tikzpicture}[vertex/.style={circle, 
  draw, 
  fill=black,
  inner sep=0.08cm}]
	\node [vertex] (l1) at  (1,1) {};
	\node [vertex] (l2) at  (0,2) {};
	\node [vertex] (l3) at  (0,1) {};
	\node [vertex] (l4) at  (0,0) {};

	\node (ll1) at (1,1.5) {$l_1$};
	\node (ll2) at (0,2.5) {$l_2$};
	\node (ll3) at (0,1.5) {$l_3$};
	\node (ll4) at (0,0.5) {$l_4$};
	\node (pp0) at (1,0.5) {$p_0$};

	\node [vertex] (p1) at  (2,1) {};
	\node [vertex] (p2) at  (3,1) {};
	\node [vertex] (p3) at  (4,1) {};
	\node [vertex] (p4) at  (5,1) {};

	\node (pp1) at (2,0.5) {$p_1$};
	\node (pp2) at (3,0.5) {$p_2$};
	\node (pp3) at (4,0.5) {$p_3$};
	\node (pp4) at (5,0.5) {$p_4$};

	\node [vertex] (r1) at  (6,1) {};
	\node [vertex] (r2) at  (7,2) {};
	\node [vertex] (r3) at  (8,1) {};
	\node [vertex] (r4) at  (7,0) {};

	\node (rr1) at (6,1.5) {$r_1$};
	\node (rr2) at (7,2.5) {$r_2$};
	\node (rr3) at (8,1.5) {$r_4$};
	\node (rr4) at (7,0.5) {$r_3$};
	\node (pp5) at (6,0.5) {$p_5$};
	
  \edge{l1}{l2}
  \edge{l1}{l3}

  \edge{l1}{l4}

  \edge{l1}{p1}
  \edge{p1}{p2}
  \edge{p2}{p3}
  \edge{p3}{p4}
  \edge{p4}{r1}

  \edge{r1}{r2}
  \edge{r2}{r3}
  \edge{r3}{r4}
  \edge{r4}{r1}
\end{tikzpicture}
\caption{The graph $K_{3,1} \oplus_4 C_4$.} \label{fig:oplus}
\end{figure}
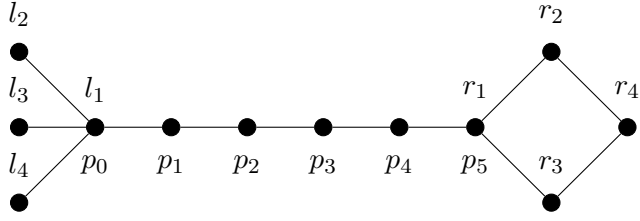

\begin{proposition}[cf. Theorem 1(5)(b) in \cite{HNP17}] \label{prop:L+R}
  With the above notation, if $q \ge s$, then 
  $$
        l(L \oplus_q R) = \begin{cases}
          1 &\text{if } m = s = 1\\
          s - 1 &\text{if } m = 1, s \ge 3\\
          m + s - 3 & \text{otherwise.}
        \end{cases}
  $$
\end{proposition}
\begin{proof}
  We write $G = L \oplus_q R$. If $m=s=1$, then $G = P_n$ where $n = q+2$, and
  hence $\genset{P_n}$ is the semigroup of order-preserving transformations
  \cite{A62}, which is aperiodic. Otherwise, we have $s \ge 3$ and we view the
  result as two matching upper and lower bounds on $l(G)$.
  \medskip

  \noindent\textbf{Lower bound.} 
  Throughout this part of the  proof, if $u, w_1, \ldots, w_t, v$ is a path in
  $G$, then we define $$(u\rightsquigarrow v)=(u\to w_1)(w_1\to
  w_2)\cdots(w_{t-1}\to w_t)(w_t\to v)\in\genset G.$$ Since this transformation
  depends on the choice of path, we will always specify the path. \medskip
  
  \noindent\textbf{Case 1: $\boldsymbol{m = 1}$.}
  We will show that there exists $\alpha \in \genset{G}$ containing the cycle
  $r_2, \dots, r_s$.  

  For each $3 \le i \le s$, we choose  $r'_i\in \{r_2, r_3\}$ such that there
  is a shortest-length path from $r_1$ to $r_i$ that avoids $r'_i$.  We also
  define
  \begin{itemize}
    \item 
      $(r_2  \rightsquigarrow p_2)$ to follow 
      the path $r_2, r_1, p_q, \ldots, p_2$,

    \item 
      $(p_{j - 1}  \rightsquigarrow r_j)$ to follow a shortest-length path
      avoiding $r_j'$ and $(r_j\rightsquigarrow p_j)$ to follow the reverse of
      such a path (but omitting the last edge) for all $3 \leq j \leq s$,

    \item
      $(r'_j \rightsquigarrow r'_{j - 1}) = (r'_j \to r_1) (r_1 \to r'_{j -
      1})$ for all $4 \leq j \leq s$, even if $r'_j=r'_{j-1}$,

    \item
      $(r_s \rightsquigarrow r'_s)$ to follow any path avoiding vertices from
      $P$.
  \end{itemize}
  It is straightforward to verify that 
  $$
    \alpha = 
    \prod_{i = 2}^{s - 1} 
    (r_i \rightsquigarrow p_i)
    \cdot 
    (r_s \rightsquigarrow r'_s) 
    \cdot 
    \prod_{j = s} ^ {4} 
    \left[ (p_{j-1} \rightsquigarrow r_j)
    (r'_j \rightsquigarrow r'_{j-1}) \right] 
    \cdot 
    (p_2 \rightsquigarrow r_3) \in\genset{G}
  $$
  contains the cycle $r_2, \dots, r_s$, as required (the second
  product is computed in descending order of the indices).  We also note that
  $l_1\alpha=l_1$.
  \medskip

  \noindent\textbf{Case 2: $\boldsymbol{m \geq 3}$.}
  As in Case 1, we may use $K_1 \oplus_q R\subseteq G$ to create
  $\alpha\in\genset{G}$ containing the cycle $r_2, \ldots, r_s$ and such that
  $l_i\alpha = l_i$ for all $i$. Similarly, we may  use $L\oplus_q
  K_1\subseteq G$ to create $\beta\in\genset{G}$ containing the cycle
  $l_2, \ldots, l_m$ and such that $r_i\beta = r_i$ for all $i$.  If 
  $(r_2 \rightsquigarrow l_3)$ and $(l_2
  \rightsquigarrow r_2)$ follow the unique shortest paths, then 
  $\gamma = \alpha \beta (r_2 \rightsquigarrow l_3) (l_2 \rightsquigarrow
  r_2)\in\genset{G}$ contains the cycle $l_3, \ldots, l_m, r_2, \ldots, r_s$.
  \medskip

\noindent\textbf{Upper bound.} 
Since $l(G_1) \le l(G_2)$ if $G_1$ is a subgraph of $G_2$, we assume without
loss of generality that $G = K_m \oplus_q K_s$. 
We define a pre-order $\preceq$ on the vertices of $G$ such that 
$a \preceq b$ if $a \in L ^*$, or $b \in R ^ *$, or $a = p_i$ and $b= p_j$ for some
$0\leq i\leq  j\leq q+1$. If $a\preceq b$ and $b\not\preceq a$, then we write
$a \prec b$ or $b\succ a$. We note that $b\not\preceq a$ implies $a\preceq b$. 
We define the sets
$$a^+ = \{b \in \nset : b \succ a \} \cup \{a\}\quad\text{and}\quad
a^- = \{b \in \nset : a \succ b \} \cup \{a\}.$$
For the remainder of the proof, we fix some $\gamma\in\genset{G}$, and we write $\gamma
= \epsilon_1 \cdots \epsilon_k$, where $\epsilon_1, \ldots, \epsilon_k$ are arcs
in $G$. We also define $\gamma_0 = \id$ and $\gamma_i = \epsilon_1 \cdots
\epsilon_i$ for all $1 \le i \le k$.
\medskip

\noindent\textbf{Case 1: $\boldsymbol{m = 1}$.}
We require the following claim.
\begin{claim} \label{claim:jumping1}
  Suppose that $m = 1$. If there are vertices $u$ and $v$ in $G$ such that
  $u\preceq v$ and $u\gamma \succ v\gamma$, then $\gamma \in \alpha \genset{G}$
  for some $\alpha \in \genset{G} \cup \{ \id \}$ such that $u^+ \alpha
  \subseteq R^*$.
\end{claim}

\begin{proof}
  If $u \in R^*$, then $u^+=\{u\}$ and so $\alpha = \id$ has the required
  properties.  Suppose that $u \notin R^*$.  Since $u \gamma \succ v \gamma$,
  there exists $i$ such that $u \gamma_i \in R ^ *$. If $j$ is the least such
  value, then $u \gamma_{j-1} = r_1$, $u^+ \gamma_{j-1} \subseteq R$, and
  $\epsilon_j = (r_1 \to r_a)$ for some $r_a \in R^*$. If we set $\alpha =
  \gamma_j$, then $u^+ \alpha = (u^+ \gamma_{j-1})\epsilon_j \subseteq R^*$.
\end{proof}

Seeking a contradiction, suppose that $\gamma$ has
a cycle of length $c$, where $c \geq s$, and let $C = \{u_1, \dots, u_c\}$
be such a cycle, where $u_1 \preceq\cdots\preceq u_c$. It is
not necessarily the case that $u_i \gamma = u_{i+1}$. Since $|C| = c > s - 1 =
|R ^ *|$, $C$ is not contained in $R^*$ and so $u_1 \prec u_i$ for all $i$. 
This gives $C\subseteq u_1^+$.
Let $j$ be such that $u_j\gamma=u_1$.  Then $u_1\prec u_j$ but $u_1\gamma\succ u_j\gamma$.
So by Claim~\ref{claim:jumping1},  $c = |C \gamma|
\le |u_1^+ \gamma| \le |R^*|=s-1$, which is the desired contradiction.  \medskip

  \noindent\textbf{Case 2: $\boldsymbol{m \geq 3}$.}
For the sake of obtaining a contradiction, suppose that $\gamma$ has a cycle of length at least $m + s - 2$. Let $C =
\{u_1, \dots, u_c\}$ be such a cycle, sorted so
that $u_1 \preceq \cdots\preceq u_c$.  Again it is not
necessarily the case that $u_i \gamma = u_{i+1}$.  We note that $u_c \notin L$
(since $c \ge m+1$) and $u_1 \notin R$ (since $c \ge s+1$), whence $u_c \succ
u_1$.  

We say that a vertex $v$ of $G$ is of
\textit{$L$-type} if there is $\beta \in \genset{G} \cup \{ \id \}$ such that
$\gamma \in \beta \genset{G}$ and $v^- \beta \subseteq L^*$. Similarly, we
say that $v$ is of \textit{$R$-type} if there is $\alpha \in \genset{G} \cup
\{ \id \}$ such that $\gamma \in \alpha \genset{G}$ and $v^+ \alpha \subseteq
R^*$. 

\begin{claim} \label{claim:jumping2}
  Suppose that $m\ge 3$.  If there are vertices $u$ and $v$ in $G$ such that
  $u\preceq v$ and $u\gamma \succ v\gamma$, then either
  $u$ is  of $R$-type, or $v$ is of $L$-type.
\end{claim}
\begin{proof} 
  As in the proof of Claim~\ref{claim:jumping1}, if $u \in R^*$, then 
  $\alpha = \id$ witnesses that $u$ is of $R$-type. Similarly, if 
  $v \in L^*$ then $\beta = \id$ shows that $v$ is of $L$-type. 
  
  Suppose that $u \notin R^*$ and $v \notin L^*$.  As before, for some $i\in
  \{1, \ldots, k\}$,  $u\gamma_i$ and $v\gamma_i$ both belong to $L^*$ or
  $R^*$.  Suppose that both $u\gamma_i$ and $v\gamma_i$ belong to $R^ *$ before
  they both belong to $L ^ *$; the case when they both first belong to $L^*$ is
  symmetric.  Let $i$ be the least value such that $u\gamma_i, v\gamma_i\in R ^
  *$.  Then for all $j < i$, $u\gamma_j \preceq v\gamma_j$ and either
  $u\gamma_j \notin R^*$ or $v\gamma_j \notin R^*$.  If $u\gamma_j \in R ^*$
  for some $j < i$, then since $u\gamma_j\preceq v\gamma_j$, it follows that
  $v\gamma_j \in R ^*$. Hence $i$ is the least value such that $u\gamma_i\in
  R ^ *$.

  We will show that $u\gamma_j \not\in L^*$ for all $0 \leq j < i$. Seeking a
  contradiction, suppose that $u\gamma_j\in L^*$ for some $0\leq j < i$, and
  let  $b = \max\{j : j < i,\ u\gamma_j \in L^*\}$. The vertices $\nset$ of $G$ can be partitioned into two parts:
  $$A = (\{u\gamma_b\} \cup
  P^* \cup R^*) \gamma_b^{-1} \qquad\text{and}\qquad B
  =  (L^* \setminus \{u\gamma_b\}) \gamma_b^{-1}.$$
  Let $x \in A$. If $x \gamma_b = u \gamma_b$, then $x \gamma_i = u \gamma_i
  \in R^*$. Otherwise, $u \gamma_b \prec x \gamma_b$.
  By maximality of $b$ and minimality of $i$, we have $\epsilon_{b+1}=(u\gamma_b\to l_1)$, $\epsilon_i=(r_1\to u\gamma_i)$, $u\gamma_j\in P^*$ for all $b<j<i$, and $u\gamma_j\preceq x\gamma_j$ for all $b<j<i$.  It follows that $x\gamma_i\in R^*$.
  Therefore, $A\gamma_i \subseteq R ^*$, and so  
  $$c \le \rk(\gamma) \le |B \gamma_b| + |A \gamma_i| \le (m-2) + (s-1) = m + s
  - 3,$$
  which contradicts the fact that $c > m+s-3$. 

  We conclude that $u\gamma_j \in P^*$ for all $j < i$, and by the argument
  concluding the proof of Claim~\ref{claim:jumping1}, we obtain $u^+\gamma_i
  \subseteq R^*$, so that $u$ is of $R$-type.
\end{proof}

\begin{claim} \label{claim:jump}
  There exist $u, v \in C$ such that $u\preceq v$, $u_c\preceq v$, and $u\gamma
  \succ v\gamma$.  There also exist $u', v' \in C$ such that $u'\preceq v'$,
  $u' \preceq u_1$, and $u'\gamma \succ v'\gamma$.  
\end{claim}

\begin{proof}
  If $u_c \in R^*$, then since $C$ intersects $R^*$ but is not contained in
  $R^*$, there exist $u, v \in C$ such that $u \notin R^*$,  $u \gamma \in
  R^*$, $v \in R^*$, and $v \gamma \notin R^*$.  Then $u$ and $v$ have the
  required properties.

  If $u_c \notin R^*$, then $u_c \succ u_i$ for all $1 \le i \le c-1$. In
  particular, if $u_c=u_j\gamma$, then $u = u_j$ and $v = u_c$ have the
  required properties.

The proof of the existence of $u'$ and $v'$ is symmetrical.
\end{proof}

We now write $X = \set{u_i \in C}{\exists j,\ u_i\preceq u_j,\ u_i\gamma \succ
u_j\gamma}$, and note that $X\not=\varnothing$ by Claim \ref{claim:jump}.
We enumerate $X = \{x_1, \dots, x_d\}$ such that $x_1 \preceq\cdots\preceq x_d$.   For each $1\leq i\leq d$,
let $y_i$ be an element of $C$ such that $x_i\preceq y_i$ and $x_i\gamma
\succ y_i\gamma$; we also assume that $y_i$ is maximal with respect to this property: that is, if $v\in C$ is such that $x_i\preceq v$ and $x_i\gamma \succ v\gamma$,
then $v \preceq y_i$.  Note that $\{y_1, \ldots, y_d\}$ is not necessarily
sorted according to the pre-order.  If $y_M$ is a maximal element of $\{y_1,
\ldots, y_d\}$ with respect to $\preceq$, then Claim~\ref{claim:jump} indicates
that $u_c \preceq y_M$.  We also have $x_1 \preceq u_1$, since $x_1\preceq u'$, where $u'$ is as in Claim \ref{claim:jump}.

\begin{claim} \label{claim:jumping3}
  There exists $1 \le a < d$ such that $y_1, \dots, y_a$ are all of $L$-type
  and $x_{a+1}, \dots, x_d$ are all of $R$-type. Moreover, for all $i > a \ge
  j$, $x_i \succ y_j$.
\end{claim}
\begin{proof} 
  We shall prove a sequence of facts about the set $X$, the last two of which give the claim.
  \medskip

  \textbf{(a).} If $x_i$ is of $R$-type, and if $x_i \preceq y_j$, then $y_j$
  is not of $L$-type.
  \medskip

  \noindent Suppose to the contrary that we have the following: $x_i \preceq
  y_j$; $\alpha, \beta \in \genset{G} \cup \{ \id \}$; $\gamma \in \alpha
  \genset{G}$ and $\gamma \in \beta \genset{G}$; $x_i^+ \alpha \subseteq R^*$
  and $y_j^- \beta \subseteq L^*$. We then have $x_i \notin L^*$, for otherwise
  $\nset = L^* \cup x_i^+$ and $$ c \le \rk(\gamma) \le |L^*
  \setminus \{ x_i \}| + |x_i^+ \alpha| \le (m-2) + (s-1) = m + s - 3, $$ a
  contradiction.  Similarly, we have $y_j \notin R^*$. Thus 
  $$ x_i, y_j \in P^*, \quad \nset = x_i^+ \cup y_j^-, \quad x_i \in
  x_i^+ \cap y_j^-.  $$ 
  Denoting $S = \{ w \in y_j^- : w \gamma = x_i \gamma \}$, we have $$ c \le
  \rk(\gamma) \le |(x_i^+ \cup S) \gamma| + |(y_j^- \setminus S) \gamma| =
  |x_i^+ \gamma| + |y_j^- \gamma| - 1 \le (s-1) + (m-1) - 1 = m + s - 3, $$ a
  contradiction.
  \medskip

  \textbf{(b).} For every $i$, $x_i$ is of $R$-type if and only if $y_i$ is not of $L$-type.
  \medskip

  \noindent Apply (a) with $i=j$ and combine with Claim~\ref{claim:jumping2}.
  \medskip

  \textbf{(c).} If $x_i$ is of $R$-type, then so too are $x_{i+1},\ldots,x_d$.
  \medskip

  \noindent If $x_i$ is of $R$-type and $i < j$, then because $x_i\preceq x_j\preceq y_j$, (a) says that $y_j$ is not of $L$-type and (b) in turn says that $x_j$ is of $R$-type.
  \medskip

  \textbf{(d).} $y_1$ is of $L$-type and $x_M$ is of $R$-type.
  \medskip

  \noindent We prove that $x_1$ is not of $R$-type, which by (b) implies that $y_1$ is of $L$-type. Suppose that $x_1^+  \alpha \subseteq R^*$ for some
  $\alpha\in\genset{G}\cup\{\id\}$ with $\gamma\in\alpha\genset{G}$. If $x_1 \in
  L^*$, then $x_1^+ = \{x_1\} \cup P^* \cup R^*$ and hence $ c \le \rk(\gamma)
  \le |L^* \setminus \{x_1\}| + |(\{x_1\} \cup P^* \cup R^*)\alpha| \le (m-2) +
  (s-1) = m + s - 3 $, a contradiction.  If $x_1 \notin L^*$, then $x_1 = u_1$
  (since $x_1 \preceq u_1$ and $u_1 \not\in R^*$) and hence $C \subseteq x_1^+$, so
  that $c=|C\gamma| \le |x_1^+ \alpha| \le s-1$, a contradiction. The proof for $x_M$ is
  symmetrical.
  \medskip

  \textbf{(e).} There exists $1 \le a < d$ such that $y_1, \dots, y_a$ are all of
  $L$-type and $x_{a+1}, \dots, x_d$ are all of $R$-type.
  \medskip

  \noindent This follows from combining (b), (c) and (d), with $a=\max\set{i}{y_1,\ldots,y_i\text{ are of $L$-type}}$.
  \medskip

  \textbf{(f).} For all $i > a \ge j$, $x_i \succ y_j$.
  \medskip

  \noindent By (e), $x_i$ is of $R$-type, and $y_j$ of $L$-type.  It follows from (a) that $x_i \succ y_j$.
\end{proof}

We now partition $C$ into two parts $A$ and $B$ defined by
$$A = \{ u \in C : x_{a+1} \succ u \} \qquad\text{and}\qquad
B = \{ v \in C : x_{a+1} \preceq v \}.$$
Note that $A$ and $B$ are both non-empty: for example, $x_{a+1}\in B$ and $y_a\in A$.
Since $C$ is a cycle of $\gamma$, $S \gamma \ne S$ for any non-empty proper 
subset $S$ of $C$. In particular, $A \gamma \ne A$ and $B \gamma \ne B$, and so 
there exist $u\in A$ and $v\in B$ such that $u \gamma \in B$ and $v \gamma \in A$.
It follows that $u\preceq v$ and $u\gamma \succ v\gamma$, and hence
$u = x_j$ and $v \preceq y_j$ for some $j \le
a$. But then $x_{a+1}\preceq v\preceq y_j$, which
contradicts Claim~\ref{claim:jumping3}.
\end{proof}

\subsection{Classification results}
In this subsection, we give a classification of the connected
graphs $G$ for which $l(G)$ is equal to $1$, $2$ or $n-1$.  From this, and in
light of equation \eqref{eq:lS1...r}, it is easy to deduce such classifications
for arbitrary graphs $G$.  We also consider the computational complexity of
determining whether a given graph $G$ satisfies $l(G)\leq k$.  

The classification of graphs with $l(G) = 1$ or $l(G) = 2$ is based on the
following family of graphs. The graph $Q_n$ for $n \ge 3$ is obtained by adding
the edge $\{n-2, n\}$ to the path $P_n$, so that the three last vertices form a
triangle (and indeed $Q_3 = K_3$). The graph $R_n$ for $n \ge 4$ is obtained by
removing the edge $\{n-1, n\}$ from $Q_n$, so that the last four vertices form
the star graph $K_{3,1}$ (and indeed $R_4 = K_{3,1}$). The graphs $Q_6$ and $R_6$ are illustrated
in Figure~\ref{fig:Q6R6}.

\begin{figure}[ht]
\centering
\begin{tikzpicture}[vertex/.style={circle, 
  draw, 
  fill=black,
  inner sep=0.08cm}]
	\node [vertex] (q1) at  (0,1) {};
	\node [vertex] (q2) at  (1,1) {};
	\node [vertex] (q3) at  (2,1) {};
	\node [vertex] (q4) at  (3,1) {};
	\node [vertex] (q5) at  (4,2) {};
	\node [vertex] (q6) at  (4,0) {};
	
	\node [vertex] (r1) at  (7,1) {};
	\node [vertex] (r2) at  (8,1) {};
	\node [vertex] (r3) at  (9,1) {};
	\node [vertex] (r4) at  (10,1) {};
	\node [vertex] (r5) at  (11,2) {};
	\node [vertex] (r6) at  (11,0) {};
	
  \edge{q1}{q2}
  \edge{q2}{q3}
  \edge{q3}{q4}
  \edge{q4}{q5}
  \edge{q5}{q6}
  \edge{q4}{q6}

  \edge{r1}{r2}
  \edge{r2}{r3}
  \edge{r3}{r4}
  \edge{r4}{r5}
  \edge{r4}{r6}
\end{tikzpicture}
\caption{The graphs $Q_6$ (left) and $R_6$ (right).} \label{fig:Q6R6}
\end{figure}
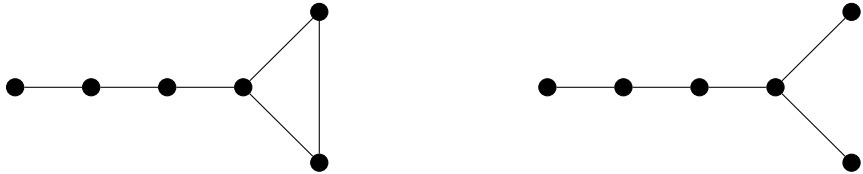

A number of other graphs, pictured in Figure~\ref{fig:bull}, will feature in
the proofs. It can be shown, using GAP~\cite{GAP4} for instance, that if $G$ is
the bull graph or the E-graph, then $l(G) = 3$ and that $l(\theta_0) = 6$.
\begin{figure}[!htp]
\centering
\begin{tikzpicture}[vertex/.style={circle, 
  draw, 
  fill=black,
  inner sep=0.08cm}]

	\node [vertex] (1) at (0,3) {};
	\node [vertex] (2) at (2,3) {};
	\node [vertex] (3) at (0,2) {};
	\node [vertex] (4) at (2,2) {};
	\node [vertex] (5) at (1,1) {};
	\node () at (3,1) {};
	\node () at (-1,1) {};

 \draw (1) -- (3);
 \draw (2) -- (4);
 \draw (3) -- (4);
 \draw (4) -- (5);
 \draw (5) -- (3);
\end{tikzpicture}
 \hspace{1 cm}
\begin{tikzpicture}[vertex/.style={circle, 
  draw, 
  fill=black,
  inner sep=0.08cm}]

	\node [vertex] (1) at (1,2) {};
	\node [vertex] (2) at (0,2) {};
	\node [vertex] (3) at (0,1) {};
	\node [vertex] (4) at (0,0) {};
	\node [vertex] (5) at (1,0) {};
	\node [vertex] (6) at (1,1) {};
	\node () at (2,1) {};
	\node () at (-1,1) {};

 \draw (1) -- (2);
 \draw (2) -- (3);
 \draw (3) -- (4);
 \draw (4) -- (5);
 \draw (3) -- (6);
\end{tikzpicture}
 \hspace{1 cm}
\begin{tikzpicture}[vertex/.style={circle, 
  draw, 
  fill=black,
  inner sep=0.08cm}]

	\node [vertex] (1) at (0:1) {};
	\node [vertex] (2) at (60:1) {};
	\node [vertex] (3) at (120:1) {};
		
	\node [vertex] (4) at (180:1) {};
	\node [vertex] (5) at (240:1) {};
	\node [vertex] (6) at (300:1) {};

	\node [vertex] (7) at (0,0) {};
	\node () at (2,1) {};
	\node () at (-2,1) {};

	\draw (1) -- (2);
	\draw (2) -- (3);
	\draw (3) -- (4);
	\draw (4) -- (5);
	\draw (5) -- (6);
	\draw (6) -- (1);
	
	\draw (1) -- (7);
	\draw (4) -- (7);
\end{tikzpicture}
\caption{The bull graph (left), E-graph (centre) and $\theta_0$ graph (right).} 
\label{fig:bull}
\end{figure}
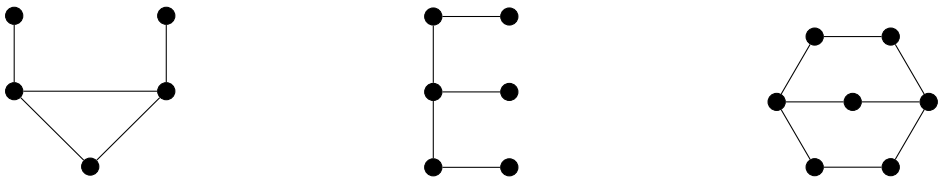

\begin{proposition} \label{prop:l=1}
  Let $G$ be a connected graph. Then $l(G) = 1$ if and only if
  $G$ is a path.
\end{proposition}

\begin{proof}
  Since $\genset{P_n}$ is the semigroup of order-preserving transformations of
  $\nset$, it is aperiodic. Conversely, suppose that $l(G) = 1$. By
  Lemmas~\ref{lem:l(Cn)},~\ref{cor:l_minor} and~\ref{lem:degree}, $G$
  is a tree with maximum degree $2$, in other words, $G$ is a path.
\end{proof}

Proposition~\ref{prop:H-trivial} easily follows from Proposition~\ref{prop:l=1}
and equation \eqref{eq:lS1...r}.

\begin{proposition} \label{prop:l=2}
  Let $G$ be a connected graph. Then $l(G) = 2$ if and only if $G$
  is $Q_n$~$(n\geq3)$ or $R_n$ $(n\geq4)$.
\end{proposition}
\begin{proof}
  Note that $Q_n=K_1\oplus_{n-4}K_3$ and $R_n=K_1\oplus_{n-4}K_{2,1}$.  It
  follows from Proposition~\ref{prop:L+R} that $l(Q_n) =  l(R_n) = 2$ for
  $n\geq7$; this can also be verified for $n\leq6$, using GAP~\cite{GAP4}. This
  part of the proof also follows from Lemma 15 in~\cite{HNP17}.

  Conversely, suppose that $l(G) = 2$; the case $n \le 3$ is easy so let us
  assume $n \ge 4$. By Lemmas~\ref{lem:l(Cn)} and \ref{cor:l_minor},~$G$ does
  not have any cycle of length $4$ or more.  By Lemma~\ref{lem:degree}, $G$ has
  no vertices of degree greater than~$3$.

  \begin{claim}
    $G$ has exactly one vertex of degree $3$.
  \end{claim}

  \begin{proof}
    If $G$ has no vertex of degree $3$, then it is a path and $l(G) = 1$, or it
    is a cycle and $l(G) = n-1 \ge 3$. Thus, $G$ has a vertex of degree 3,
    say $x_1$, with neighbours $x_2$, $x_3$ and~$x_4$. First, suppose that $x_2$ also has
    degree 3. If $x_3$ and $x_4$ are both neighbours of $x_2$, then $G$ has the
    cycle $x_1,x_3,x_2,x_4$. If $x_2$ is adjacent to $x_3$ and to another vertex, say $x_5$,
    then $G$ contains a bull. Thus, $x_2$ is not adjacent to either $x_3$ or $x_4$,
    and instead is adjacent to $x_5$ and $x_6$, say, in which case, $G$ contains a tree
    with leaves $x_3$, $x_4$, $x_5$, and $x_6$, so Lemma~\ref{lem:degree}(ii) applies. 
    So vertex $x_2$ does not have degree $x_3$ and, similarly, neither do vertices $x_3$ and $x_4$.
    Second, suppose that $G$ contains another vertex of degree $3$, say $u$,
    that is not a neighbour of $x_1$.  There is a path from $u$ to $x_1$, and we
    may assume this goes through vertex $x_2$; then by contracting the path from
    $x_2$ to $u$, and applying Lemma \ref{cor:l_minor}, we get back to the first case.  
  \end{proof}

  We now split the rest of the proof into two cases. First, if $G$ is a tree,
  then $G = R_n$ or $G$ has the E-graph as a subgraph. The latter case would
  yield $l(G) \ge 3$, hence $G = R_n$. Second, if $G$ is not a tree, then $G$
  has a triangle, say induced by the vertices $a,b,c$. One of them must
  be the vertex of degree 3, say $a$, and the other two have degree $2$.
  Then $G = Q_n$.
\end{proof}

On the other extreme, we have the following classification. Recall that a
graph $G$ is \textit{non-separable} if for every pair of vertices $u, v \in
\nset$, there are at least two vertex-disjoint paths from $u$ to
$v$. 

\begin{proposition}\label{prop:l=n-1}
  Let $G$ be a connected graph. Then $l(G) = n-1$ if and only if
  $G = K_2$ or $G$ is non-separable and not odd bipartite. 
\end{proposition}
\begin{proof}
  The case $n \le 3$ being easily checked, we assume $n \ge 4$ throughout the
  proof.

  Let $G$ be non-separable. Recall the puzzle group $\Gamma_G(v)$ from
  \cite{W74}, obtained as follows. First of all, create a hole at any vertex
  $v$. Then repeatedly slide a vertex $a$ into the hole at vertex $b$, where
  $a$ is adjacent to $b$; this moves the hole to $a$. Whenever the hole goes
  back to $v$, this yields a permutation of $\nset \setminus \{v\}$. The
  (abstract) group does not actually depend on $v$. Clearly, creating the hole
  at $v$ can be done by using any arc $(v \to u)$ where $u$ is a neighbour of
  $v$, and then sliding a vertex $a$ to the hole in $b$ is equivalent to using
  the arc $(a \to b)$. Therefore, for any initial hole $v$ and any $g \in
  \Gamma_G(v)$ acting on $\nset \setminus \{v\}$, there exists
  $\alpha \in \genset{G}$ such that $u g = u \alpha$ for all $u \ne v$. 

  ($\Leftarrow$) We have already noted that $l(\theta_0)=6$, and that
  $l(C_n)=n-1$.  Let $G$ be non-separable and neither a cycle nor the graph
  $\theta_0$. According to \cite[Theorem~2]{W74}, $\Gamma_G(v) = \Alt_{n-1}$ if
  $G$ is bipartite and $\Gamma_G(v) = \Sym_{n-1}$ otherwise. Therefore $l(G) =
  n-1$ if $G$ is non-separable and not odd bipartite, or $l(G) \ge n-2$ if $G$
  is non-separable and odd bipartite.

  ($\Rightarrow$) We prove the contrapositive.
  Suppose first that $G$ is non-separable and odd bipartite.  To obtain a
  contradiction, suppose that $\beta \in \genset{G}$ has $l(\beta) = n-1$. Due
  to the form of $\beta$, there exist $u$ and $v$ such that $u \beta = v
  \beta$, $u \beta^{-1} \ne \varnothing$ and $v\beta^{-1} = \varnothing$ (note
  that $\beta$ acts as a cyclic permutation $\pi$ on $\nset
  \setminus \{v\}$). Since $(v \to u) \beta = \beta$, we can assume that the
  first arc in any word expressing $\beta$ is $(v \to u)$. This corresponds to
  creating a hole in $v$, and then expressing $\pi$ as a member of
  $\Gamma_G(v)$, which is impossible since $\pi$ is an odd permutation while
  $\Gamma_G(v) = \Alt_{n-1}$.

  Now suppose that $G$ is separable.  So there exist $L, R \subseteq \nset$ and $v\in\nset$ such that $2 \le |L| \le |R|$, $L
  \cap R = \{v\}$, and for any edge $\{l,r\}$ of $G$ with $l \in L$ and $r \in
  R$ we have $v \in \{l,r\}$; see for example \cite[Theorems 5.1 and 5.2]{BM08}. Then $G$ is a minor of $L \oplus_n R$, which is
  itself a minor of $K_m\oplus_nK_s$, where $m=|L|$ and $s=|R|$. By 
  Lemma~\ref{cor:l_minor} and Proposition~\ref{prop:L+R},
  $$
  l(G)\leq l(K_m\oplus_nK_s) = \begin{cases}
    1 &\text{if $m=s=2$}\\
    s-1 &\text{if $m=2<s$}\\
    m+s-3 &\text{if $m,s>2$.}
  \end{cases}
  $$
  (Note that when $m=2$, for example, we have
  $K_2\oplus_nK_s=K_1\oplus_{n+1}K_s$.)  Thus, in all cases, $l(G)\leq n-2$.
\end{proof}

We remark that the proof of Proposition~\ref{prop:l=n-1} (in conjunction with
Proposition~\ref{prop:l=n-1} itself) indicates that if $G$ is non-separable and
odd bipartite, then $l(G) = n-2$.  

A classification of graphs $G$ such that
$l(G) \le k$ for arbitrary $k$ seems beyond reach at the moment. However, since
these graphs form a minor-closed class, we can determine whether $l(G) \le k$
in time $O(n^2)$ \cite{KKR12}. We show that in fact this can be done in linear
time.

\begin{theorem} \label{th:l_linear_time}
  For any fixed $k$, deciding whether a connected graph $G$,
  given as an adjacency list, satisfies $l(G) \le k$ can be done in $O(n)$
  time.
\end{theorem}
\begin{proof}
  Let us refer to a maximal path in $G$ consisting of vertices of degree $2$ as
  a branch. If a branch does not belong to a non-separable block, then $G = L
  \oplus_q R$, where the branch is the path in the middle. We say that a branch
  is terminal if $L = K_1$ and non-terminal otherwise. We shall use the same
  notation as for Proposition~\ref{prop:L+R}.

  The result is clear for $k=1$ (Proposition~\ref{prop:l=1}), so suppose $k \ge
  2$. The algorithm goes as follows.
  \begin{enumerate}

    \item
      \label{it:1} If $n \le (k+2)(k+1) (2k-1)$, solve by brute force, i.e.\  by
      enumerating all elements of $\genset{G}$.

    \item
      \label{it:2} If $G$ is a path, then return \textit{Yes}.

    \item
      \label{it:3} If $G$ has a vertex of degree at least $k+2$, then return
      \textit{No} (Lemma~\ref{lem:degree}).

    \item
      \label{it:4} If $G$ has at least $4k-1$ vertices of degree not $2$, then
      return \textit{No} (Lemma~\ref{lem:degree}(iii)).

    \item
      \label{it:5} If $G$ has a non-separable block of size at least $k+3$,
      then
      return \textit{No} (Proposition~\ref{prop:l=n-1} and the remark after its
      proof).

    \item
      \label{it:6} Let $P$ be the longest branch of $G$. If $P$ is terminal and
      has length at most $n - k - 3$, or if $P$ is non-terminal and has length
      at most $n - k - 4$, then return \textit{No}. Otherwise, return \textit{Yes}.
  \end{enumerate}

  If the first five properties are not satisfied, then the number of vertices
  of degree $2$ is at least 
  $$
        n - t \ge  (k+2)(k+1) (2k-1) + 1 - (4k-2) > (k+1)^2 (2k-1).
  $$
  On the other hand, the number of branches is at most $(k+1)t/2 \le
  (k+1)(2k-1)$. Thus the longest branch $P$ of $G$ has length $q \ge k + 2$. 

  First, suppose that $P$ is terminal, i.e.\  $G = L \oplus_q R$ with $m = 1$ and
  $s \ge 3$. If $q \le n- k - 3$, then $s - 1 = n - q - 2 \ge k+1$ and $l(G)
  \ge l(G') = k + 1$, where $G'$ is the subgraph of $G$ induced by $L \cup P
  \cup \{r_1, \dots, r_{k+2}\}$. Otherwise, $s \le k+1$ hence $q \ge s$ and
  $l(G) = s - 1 \le k$.

  Second, suppose that $P$ is non-terminal: i.e.\  $G = L \oplus_q R$ with $s \ge
  m \ge 3$. If $q \le n- k - 4$, then $m + s - 3 = n - q - 3 \ge k+1$. Let 
$$
    \mu = \min \left\{ m, \left\lfloor \frac{k+4}{2} \right\rfloor \right\}
    \qquad\text{and}\qquad
        \sigma = k + 4 - \mu.
$$
  We then have 
  $$
        \sigma + \mu - 3 = k + 1, \quad q \ge k + 2 \ge \sigma \ge \mu \ge 3,
        \quad m \ge \mu, \quad s \ge \sigma
  $$
  and $l(G) \ge l(G') = k + 1$, where $G'$ is the subgraph of $G$ induced by
  $\{l_1, \dots, l_\mu\} \cup P \cup \{r_1, \dots, r_\sigma\}$. Otherwise, $m +
  s - 3 \le k$ hence $q \ge s$ and $l(G) = m + s - 3 \le k$.

  Step~\ref{it:1} runs in $O(1)$ time; properties~\ref{it:2} to~\ref{it:4} are
  decidable in time $O(n)$. If the first four properties are not satisfied, the
  number $m$ of edges of $G$ is at most $\frac{1}{2}(k+1)t + n-t \le 2n$. Then
  the following steps, which run in $O(n + m)$ (an algorithm to find the
  non-separable blocks in linear time is given in \cite{HT73}), actually run in
  $O(n)$ time.
\end{proof}

\section{Properties related to Green's relations} \label{sec:Green}

In this section we characterise some semigroup theoretic properties of
$\genset{D}$ in terms of certain digraph theoretic properties of $D$.  In
Proposition~\ref{prop:H-trivial}, we classified the digraphs $D$ for which
$\genset{D}$ is $\H$-trivial.  The purpose of this section is to give analogous
classifications for Green's $\R$-, $\LL$- and $\J$-relations in
Propositions~\ref{prop-r-trivial},~\ref{prop-l-trivial}
and~\ref{prop-j-trivial}, respectively.

The proof of \cite[Proposition 2.4.2]{H95} gives the following.

\begin{lemma}\label{lem:RL_restriction}
  Let $T$ be a subsemigroup of a semigroup $S$, let $a,b\in T$, and suppose
  $a,b$ are regular in~$T$.  Then the following hold:
  \begin{enumerate}[\rm (i)]
    \item 
      $a, b$ are $\R$-related in $T$ if and only if they are $\R$-related
      in $S$;
    \item 
      $a, b$ are $\LL$-related in $T$ if and only if they are $\LL$-related
      in $S$.
  \qed
  \end{enumerate}
\end{lemma}

Recall that two elements of $\Tran_n$ are $\R$-, $\LL$-, or $\J$- related if
and only if they have the same kernel, image, or rank, respectively.
\begin{lemma} \label{lem-r-related-idempotents}
  Let $D$ be a digraph. If $D$ contains a cycle and $(a
  \to b)$ is an arc in that cycle, then $(b\to a) \in \genset{D}$ and
  $(a\to b)\R(b \to a)$.
\end{lemma}
\begin{proof}
  As noted earlier, it follows from \cite[Lemma 2.3]{YY06} that $(b \to a)$
  belongs to $\genset{D}$.  Since $(a\to b)$ and $(b \to a)$ are idempotents, and
  hence regular, and they have equal kernels, it follows follows from
  Lemma~\ref{lem:RL_restriction} that $(a\to b)\R (b \to a)$.
\end{proof}

\begin{proposition}\label{prop-r-trivial}
  Let $D$ be a digraph. Then $\genset{D}$ is $\R$-trivial if and only if $D$ is
  acyclic.
\end{proposition}
\begin{proof}
  ($\Rightarrow$) 
  It follows immediately from Lemma~\ref{lem-r-related-idempotents} that 
  if $D$ contains a cycle, then it is not $\R$-trivial, and so the
  contrapositive of this implication holds.

  ($\Leftarrow$) 
  Again we prove the contrapositive. Suppose that $\alpha,\beta\in\genset{D}$
  are such that $\alpha\not=\beta$ and $\alpha\R\beta$.  Then
  there exist $\gamma,\delta\in\genset{D}$ such that 
  $\alpha\gamma=\beta$ and $\beta\delta=\alpha$, and there is
  $i\in\nset$ with $i\alpha\not=i\beta$.  Hence
  $i\alpha\gamma = i\beta \not=i\alpha$ and $i\alpha=i\alpha\gamma\delta$. The
  former implies that $D$ contains a
  non-trivial path from $i\alpha$ to $i\alpha\gamma$, and the latter that $D$
  contains a path from $i\alpha\gamma$ to $i\alpha$.  Thus $D$ contains a cycle.
\end{proof}

\begin{proposition}\label{prop-l-trivial}
  Let $D$ be a digraph. Then $\genset{D}$ is $\LL$-trivial if and only if the
  following hold:
  \begin{enumerate}[\rm (i)]

    \item 
      the out-degree of every vertex in $D$ is at most $1$; and

    \item 
      $D$ contains no cycles of length greater than $2$.
  \end{enumerate}
\end{proposition}
\begin{proof}
  ($\Rightarrow$) 
  We prove the contrapositive (i.e.\   that if either (i) or (ii) is not true,
  then $\genset{D}$ is not $\LL$-trivial).  If there are distinct arcs $\alpha =
  (a \to b)$ and $\beta = (a\to c)$ in $D$, then $\alpha$ and $\beta$ are
  regular, and have the same image, and so $\alpha\LL\beta$. If $D$ contains a
  cycle of length greater than $2$, then $\genset{D}$ is not $\H$-trivial, by
  Proposition~\ref{prop:H-trivial}, and hence it is not $\LL$-trivial.

  ($\Leftarrow$)
  Suppose that both (i) and (ii) both hold.  We begin by making some
  observations about the elements of $\genset{D}$ and their action on the
  vertices of $D$.

  Suppose that $x_0 \in \nset$ is an arbitrary vertex of $D$ with
  out-degree $1$.  By the assumptions on the structure of $D$, there is a
  unique path  
  \begin{equation}\label{eq:*} 
    x_0 \to x_1 \to \cdots \to x_{k-1}\to x_k 
  \end{equation} 
  in $D$  starting at $x_0$, and where $x_{k}$ has out-degree $0$ or $(x_{k}
  \to x_{k-1})$ is an arc in $D$. Since there are no vertices in $D$ with
  out-degree exceeding $1$, it follows that $(x_i \to x_{i + 1})$ is the only
  arc in $D$ starting at $x_i$ for every $i$.  So, if $\gamma \in \genset{D}$,
  then 
  \begin{equation} \label{eq-1-way}
    x_s\gamma = x_t\quad\text{and}\quad s\leq t\quad \text{for all}\quad 
    s \leq k -1.
  \end{equation}
  First, we will show that 
  \begin{multline}\label{eq-2-way}
    \text{if }
    \gamma=\gamma_0\easty{x_0}{x_1}\gamma_1\easty{x_1}{x_2}\gamma_2\cdots
    \gamma_{r-1}\easty{x_{r-1}}{x_r}\gamma_r \text{ where }
    \gamma_0,\gamma_1,\ldots,\gamma_r\in\genset{D},\\
    \text{then } x_0\gamma = x_s\ 
    \text{where either}\ s\geq r\ \text{or } s = k -1\ \text{and }r = k
  \end{multline}
  for all $0\leq r\leq k$.

  We proceed by induction on $r$.  If $r = 0$, then $\gamma = \gamma_0$ and
  $x_0\gamma = x_l$ for some $l\geq 0$ by \eqref{eq-1-way}.  If $r > 0$, then
  by induction there exists $l\geq r - 1$ such that
  $$x_0\gamma_0\easty{x_0}{x_1}\gamma_1\easty{x_1}{x_2}\gamma_2\cdots
  \gamma_{r-1} = x_l.$$ 
    Suppose first that $l \geq r$.  Then $x_0\gamma = x_l\easty{x_{r-1}}{x_r}\gamma_r= x_l\gamma_r$. If $l\leq k-1$, then \eqref{eq-1-way} gives $x_l\gamma_r=x_s$ for some $s\geq l\geq r$. If $l=k$, then by the form of $D$, $x_l\gamma_r=x_k\gamma_r$ can only be one of $x_k$ or $x_{k-1}$.
  On the other hand, if $l = r -1$, then 
  $x_0\gamma = x_{r-1}\easty{x_{r-1}}{x_r}\gamma_r= x_r\gamma_r$,
  and $x_r\gamma_r = x_m$ where $m\geq r$, again by~\eqref{eq-1-way}, unless $r=k$ in which case it is possible that $x_r\gamma_r=x_{k-1}$.  This completes the proof of~\eqref{eq-2-way}.

  Second, suppose that $\gamma\in \genset{D}$ and $x_0\gamma = x_r$ for some
  $r$. Since there is a unique path from $x_0$ to~$x_r$ in $D$, it follows that
  any factorisation of $\gamma$ in the arcs of $D$ must contain each of 
  $$(x_0\to x_1), (x_1\to x_2), \ldots, (x_{r -1}\to x_r)$$
  in this order. In other words,
  \begin{equation}\label{eq-3-way}
    \text{if } x_0\gamma = x_r\ \text{for some }r,\ \text{then }
    \gamma = \gamma_0\easty{x_0}{x_1}\gamma_1\easty{x_1}{x_2}\gamma_2\cdots
    \gamma_{r-1}\easty{x_{r-1}}{x_r}\gamma_r
  \end{equation}
  for some $\gamma_0,\gamma_1,\ldots,\gamma_r\in\genset{D}$. 

  We will now begin the proof of this implication in earnest.  Suppose that
  there are $\alpha,\beta\in\genset{D}$ such that $\alpha\not=\beta$. We will
  show that $\alpha$ and $\beta$ are not $\LL$-related.  Since $\alpha\not =
  \beta$, there exists $x_0\in \nset$ such that
  $x_0\alpha\not=x_0\beta$.  Since at least one of $\alpha,\beta$ does not fix
  $x_0$, it follows that the out-degree of $x_0$ is equal
  to $1$. Suppose that $x_1,\ldots,x_k$ are as in \eqref{eq:*}.  From
  \eqref{eq-1-way}, $x_0\alpha=x_r$ and $x_0\beta=x_s$ for some $r,s\geq0$.  We
  may assume without loss of generality that $r > s$.  We consider two cases
  separately.
  \bigskip

  \noindent\textbf{Case 1: $\boldsymbol{s\leq k - 2}$ or $\boldsymbol{(x_{k} \to
  x_{k-1})}$ is not an arc in $\boldsymbol{D}$.}
  By \eqref{eq-3-way}, 
  $$\alpha=\alpha_0\easty{x_0}{x_1}\alpha_1\easty{x_1}{x_2}\alpha_2\cdots
  \alpha_{r-1}\easty{x_{r-1}}{x_r}\alpha_r \qquad\text{for some }
  \alpha_0,\alpha_1,\ldots,\alpha_r\in\genset{D}.$$
  It follows that if $\gamma\in \genset{D}$ is such that $\beta=\gamma\alpha$,
  then 
  $$\beta = \gamma\alpha =
  (\gamma\alpha_0)\easty{x_0}{x_1}\alpha_1\easty{x_1}{x_2}\alpha_2\cdots
    \alpha_{r-1}\easty{x_{r-1}}{x_r}\alpha_r.$$
  But $x_0\beta = x_s$ and so \eqref{eq-2-way} implies that $s \geq r$, 
  contradicting the assumption that $s<r$. Hence $(\alpha, \beta)\not \in \LL$. 
  \bigskip

  \noindent\textbf{Case 2: $\boldsymbol{s = k - 1}$ and $\boldsymbol{(x_{k} \to
  x_{k-1})}$ is an arc in $\boldsymbol{D}$.}  
  Since $r > s$, it follows that $x_0\alpha=x_k$
  and $x_0\beta=x_{k-1}$.  By \eqref{eq-3-way},  there exist $\alpha_0,
  \alpha_1, \ldots \alpha_k\in \genset{D}$ such that 
  $$ \alpha= \alpha_0\easty{x_0}{x_1}\alpha_1\easty{x_1}{x_2}\alpha_2\cdots
  \alpha_{k-1}\easty{x_{k-1}}{x_k}\alpha_k,$$
  and we may assume without loss of generality that $\alpha_k$ does not have
  $(x_{k} \to x_{k -1})$ as a factor. Suppose that $\gamma\in
  \genset{D}$ is arbitrary. By \eqref{eq-2-way}, 
  $$x_0\gamma  \alpha_0\easty{x_0}{x_1}\alpha_1\easty{x_1}{x_2}\alpha_2\cdots
  \alpha_{k-2}\easty{x_{k-2}}{x_{k-1}}\alpha_{k-1} \in \{x_{k - 1}, x_k\}.$$
  In either case, $x_0\gamma \alpha = x_k \not = x_{k-1} = x_0\beta$ and so 
  $\beta\not=\gamma\alpha$ for any $\gamma\in\genset{D}$, which implies
  $(\alpha,\beta)\not\in\LL$.
\end{proof}

\begin{proposition}\label{prop-j-trivial}
  Let $D$ be a digraph. Then the following are equivalent:
  \begin{enumerate}[\rm (i)]

    \item
      $\genset{D}$ has at most one idempotent in every $\LL$-class and every
      $\R$-class;

    \item
      $\genset{D}$ is $\J$-trivial;

    \item
      $D$ is acyclic and the out-degree of every vertex in $D$ is at most
      $1$.
  \end{enumerate}
\end{proposition}
\begin{proof}
  (i) $\Rightarrow$ (iii).
  If $D$ had a vertex of out-degree greater than $1$, then, as in the proof of
  Proposition~\ref{prop-l-trivial},~$\genset{D}$ would contain two distinct
  $\LL$-related idempotents.

  If $D$ contained a cycle, then, by
  Lemma~\ref{lem-r-related-idempotents},~$\genset{D}$ would contain two
  distinct $\R$-related idempotents.

  (iii) $\Rightarrow$ (ii). 
  If $D$ is acyclic and the out-degree of every vertex in $D$ is at most
  $1$, then, by Propositions~\ref{prop-r-trivial}
  and~\ref{prop-l-trivial}, $\genset{D}$ is both $\R$- and $\LL$-trivial. 
  Hence $\genset{D}$ is $\J$-trivial.

  (ii) $\Rightarrow$ (i).
  Since $\genset{D}$ is $\J$-trivial, it is both $\LL$- and $\R$-trivial.  Hence
  every $\LL$-class and every $\R$-class contains exactly one element, and, in
  particular, at most one idempotent.
\end{proof}

\section{Other classical semigroup properties} \label{sec:properties}
A semigroup $S$ is called \textit{completely regular} if every element belongs
to subgroup. Equivalently, a semigroup is completely regular if and only if
every element is $\H$-related to an idempotent.  A finite semigroup $S$ is
completely regular if and only if $x\J x^2$ for all $x\in S$.  

If $D$ is any digraph with at most $2$ vertices, then $\genset{D}$ is a band.
Hence in the next two results we will assume that the number of vertices in $D$
is at least $3$.

We say a digraph $D$ is \textit{directed-bipartite} if there is a partition
of the vertices $\nset$ of $D$ into two parts $V_1$ and $V_2$ such
that every arc $(a\to b)$ of $D$ satisfies $a \in V_1$ and $b \in V_2$.

\begin{proposition} \label{prop:completely_regular}
  Let $D$ be a connected digraph with at least $3$ vertices. 
  Then the following are equivalent:
  \begin{enumerate}[\rm (i)]

    \item
      $\genset{D}$ is a band;

    \item
      $\genset{D}$ is completely regular;

    \item $D$ is directed-bipartite.
  \end{enumerate}
\end{proposition}

\begin{proof}
  (i) and (iii) are shown to be equivalent in \cite[Theorem
  2.12]{Cameron2016aa}.

  (i) $\Rightarrow$ (ii). This implication follows
  immediately, since every band is completely regular.

  (ii) $\Rightarrow$ (iii).  We prove the contrapositive.  Suppose that $D$
  contains the arcs $(a\to b)$ and $(b\to c)$, where $a,b,c\in\nset$
  are distinct, and consider $\alpha=(b\to c)(a\to b)$.  Then $\rk(\alpha)=n-1$
  and $\rk(\alpha^2)=n-2$, and so $\alpha$ and $\alpha^2$ are not $\J$-related
  in $\Tran_n$. It follows that $\alpha$ and $\alpha^2$ are not $\J$-related in
  $\genset{D}$, and so $\genset{D}$ is not completely regular.
\end{proof}

\begin{corollary}\label{prop-bands}
  Let $D$ be a connected acyclic digraph with at least $3$ vertices. 
  Then the following are equivalent:
  \begin{enumerate}[\rm (i)]

    \item
      $\genset{D}$ is a band;

    \item
      $\genset{D}$ is completely regular;

    \item
      $\genset{D}$ is regular;

    \item  
      $D$ is directed-bipartite.
  \end{enumerate}
\end{corollary}
\begin{proof}
  It suffices to prove that (iii) implies (i), so we suppose that $\genset{D}$
  is regular.
  Since $D$ is acyclic, $\genset{D}$ is $\R$-trivial by
  Lemma~\ref{prop-r-trivial}. Since $\genset{D}$ is regular, it follows that
  every $\R$-class contains an idempotent. But every $\R$-class is of size 1,
  and so every element of $\genset{D}$ is an idempotent. In other words,
  $\genset{D}$ is a band.
\end{proof}
  

There are non-acyclic digraphs $D$ such that $\genset{D}$ is regular but not a
band. For example, if $D$ is any strong tournament with $n\geq3$ vertices, then
$\genset{D} = \Sing_n$, and $\Sing_n$ is regular but not a band. 

For $n\geq1$, an \textit{$n$-fan} is a connected acyclic digraph
isomorphic to the digraph with arcs $(i, n)$ for all $i\in\{1, \ldots,
n-1\}$. A $1$-fan is just a one-vertex digraph with no arcs.  A picture of an
$n$-fan can be seen in Figure~\ref{fig-fan}.

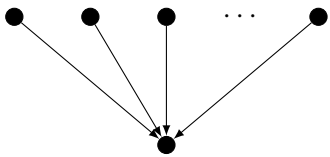
\begin{figure}[ht]
  \centering
  \begin{tikzpicture}[vertex/.style = {circle, 
    draw, 
    fill = black, 
    inner sep = 0.08cm},
    edge/.style   = {arrows = {-angle 90},
    thick}]
    \vertex{1}{2}{0} 
    \vertex{2}{0}{1} 
    \vertex{3}{1}{1} 
    \vertex{4}{2}{1} 
    \vertex{6}{4}{1}  

    \node (5) at (3, 1.7) {$\cdots$};

    \arc{2}{1}
    \arc{3}{1}
    \arc{4}{1}
    \arc{6}{1}

  \end{tikzpicture} 
  \caption{An $n$-fan.}
  \label{fig-fan}
\end{figure}

A semigroup $S$ is called \textit{inverse} if for all $x\in S$ there exists a
unique $y\in S$ such that $xyx = x$ and $yxy = y$.  It is well-known (see, for
example, \cite[Theorem 5.1.1]{H95}) that a semigroup $S$ is inverse if and only
if it is regular and its idempotents commute. The same theorem from \cite{H95}
also says that $S$ is inverse if and only if each $\R$-class and each
$\LL$-class of $S$ contains exactly one idempotent.  A \textit{semilattice} is a
semigroup of commuting idempotents.  For any set $X$, the \textit{power set}
$2^X=\{A:A\subseteq X\}$ of $X$ is a semilattice under $\cup$; the subsemigroup $2^X\setminus\{\varnothing\}$ is called the
\textit{free semilattice of degree $|X|$}.  If $X_1,\ldots,X_k$ are disjoint finite sets, then $(2^{X_1}\times\cdots\times2^{X_k})\setminus\{(\varnothing,\ldots,\varnothing)\}$ is a free semilattice of degree $|X_1|+\cdots+|X_k|$, isomorphic to $2^{X_1\cup\cdots\cup X_k}\setminus\{\varnothing\}$.

\begin{proposition}\label{prop-inverse}
  Let $D$ be a connected digraph. Then the following are equivalent:
  \begin{enumerate}[\rm (i)]
    \item
      $\genset{D}$ is a free semilattice of degree $n-1$;

    \item
      $\genset{D}$ is inverse;

    \item
      $\genset{D}$ is commutative;

    \item
      $D$ is a fan.
  \end{enumerate}
  If any of the above conditions holds, then $|\genset{D}| = 2 ^ {n
  - 1} - 1$.
\end{proposition}
\begin{proof}
  (i) $\Rightarrow$ (ii).
  Every semilattice is an inverse semigroup. 

  (ii) $\Rightarrow$ (iii). 
  Since $\genset{D}$ is inverse, it has exactly one idempotent in every $\LL$-
  and $\R$-class, and hence by Proposition~\ref{prop-j-trivial}, $\genset{D}$
  is $\J$-trivial. Thus $\genset{D}$ is a semilattice, and hence it is
  commutative.

  (iii) $\Rightarrow$ (iv). 
  Assume that $\genset D$ is commutative.  If $D$ contains distinct arcs $\alpha=(a \to b)$ and $\beta=(b\to c)$, then
  $\alpha\beta\not=\beta\alpha$, a contradiction.  If $D$
  contains distinct arcs $\gamma=(d \to e)$ and $\delta=(d\to f)$, then
  $\gamma\delta\not=\delta\gamma$, a contradiction.  Since $D$ is connected, it
  follows that $D$ is a fan.

  (iv) $\Rightarrow$ (i). 
  We may assume that the unique sink in $D$ is the vertex $n$. 
  If $S$ is any subset of $\{1, \ldots, n - 1\}$, then we define
  $\alpha_S\in\Tran_n$ by
  \begin{equation*}
    v \alpha_S = 
    \begin{cases}
      n & \text{if } v \in S,\\
      v & \text{if } v \not\in S.
    \end{cases}
  \end{equation*}
  If $S = \{s_1, \dots, s_k\}$ is not empty, then $\alpha_S = (s_1 \to n)
  \cdots
  (s_k \to n) \in \genset{D}$.
  Conversely, the arcs of $D$ commute and so any transformation in $\genset{D}$
  is of the form $\alpha_S$ for some non-empty subset $S$ of $\{1, \ldots, n -
  1\}$. If $S$ and $T$ are non-empty subsets of $\{1, \ldots, n - 1\}$, then 
  it is routine to verify that $\alpha_S \alpha_T = \alpha_{S \cup T}$. 
  It follows that the map  $\phi: 2^{\{1, \ldots, n -
  1\}}\setminus\{\varnothing\}\to\genset{D}$ defined by $S\phi = \alpha_S$
  is an isomorphism.
\end{proof}

If $D$ is a digraph with connected components $D_1,  \ldots, D_r$, then 
it follows from Proposition~\ref{prop-direct-product} that $\genset D$ is
inverse if and only if each $\genset{D_i}$ is inverse. From this we obtain the
following corollary to Proposition~\ref{prop-inverse}.

\begin{corollary}
  The number of digraphs (up to isomorphism) with $n$ vertices such that
  $\genset{D}$ is an inverse semigroup is equal to $n-1$.
\end{corollary}
\begin{proof}
  Suppose that $\genset D$ is inverse, that the connected components of $D$ are $D_1,  \ldots, D_r$, and write $d_i=|D_i|$ for each $i$.  It
  follows from Proposition~\ref{prop-inverse} that each $D_i$ is a fan, and
  each $\genset{D_i}$ is a free semilattice of degree $d_i-1$.  It follows that
  $\genset D$ is a free semilattice of degree $(d_1-1)+\cdots+(d_r-1)=n-r$.  So
  the isomorphism class of $\genset D$ is completely determined by $r$, the
  number of connected (fan) components of $D$.  Since $r$ can take any value
  from $1$ to $n-1$, the proof is complete.
\end{proof}


\section{Zeros} \label{sec:zeros}

An element $a$ of a semigroup $S$ is a \textit{left zero} if $ab=a$
for all $b\in S$.  Right zeros are defined analogously.  An element is a
\textit{zero} if it is both a left and right zero.  If a semigroup has a left
zero and a right zero, then it has a unique zero.  In this section, we obtain
necessary and sufficient conditions on a digraph $D$ so that $\genset{D}$ has
various properties associated to left or right zeros.  Some of our results
also classify the digraphs $D$ for which $\genset{D}$ consists of a single
$\R$-, $\LL$-, or $\J$-class (see Proposition~\ref{prop-trivial} for the analogous
result for the $\H$-relation).  We begin with two technical lemmas.

\begin{lemma}\label{lem:terminal1}
  If $D$ is strongly connected, then $\genset{D}$ contains every constant map.
\end{lemma}
\begin{proof} 
  Since $D$ is strongly connected, it suffices to show that $\genset{D}$
  contains any constant map.  Suppose that there exists $\alpha\in \genset{D}$
  with rank $r$, where $2\leq r \leq n - 1$. We prove that there exists
  $\beta\in \genset{D}$ with $\rk(\beta) < \rk(\alpha)$, from which the lemma
  follows. Since $\rk(\alpha) \geq 2$, there exist distinct $i, j\in
  \im(\alpha)$, and since $D$ is strongly connected, there is a path in $D$
  from $i$ to $j$. If $(i\rightsquigarrow j)$ denotes the product of the arcs
  in a path from $i$ to $j$, then it is routine to check
  that $\rk(\alpha (i\rightsquigarrow j)) < \rk(\alpha)$, as required.
\end{proof}

\begin{lemma}\label{lem:terminal2}
  If $D$ is connected, then there exists $\alpha\in\genset{D}$ such that
  $i\alpha$ belongs to a terminal component of $D$ for each $i\in\nset$.
\end{lemma}
\begin{proof}
  Suppose without loss of generality that $\{1, \ldots, k\}$ are the vertices
  of $D$ that do not belong to a terminal component of $D$. Then for every
  $i\in \{1, \ldots, k\}$ there exists a vertex $t_i$ in a terminal component
  of $D$ such that there is a path in $D$ from $i$ to $t_i$. 
  The product $\alpha = \prod_{i = 1}^{k} (i \rightsquigarrow t_i)$
  has the required property.
\end{proof}

\begin{proposition} \label{prop:zero}
  Let $D$ be a connected digraph. Then the following hold:
  \begin{enumerate}[\rm (i)]

    \item
      \label{it:left-0} $\genset{D}$ has a left zero if and only if all
      terminal components of $D$ are trivial. If this is the case, then
      $\alpha$ is a left zero of $\genset{D}$ if and only if $v \alpha$ belongs
      to a terminal component for all vertices $v$.

    \item
      \label{it:right-0} $\genset{D}$ has a right zero if and only if it has
      exactly one terminal component. If this is the case, then $\alpha$ is a
      right zero of $\genset{D}$ if and only if it is a constant map.

    \item
      \label{it:0} $\genset{D}$ has a zero if and only if it has exactly one
      terminal component $T$, which is trivial. If this is the case, then the
      zero of $\genset{D}$ is the constant map with image $T$.
  \end{enumerate}
\end{proposition}
\begin{proof}
  (\ref{it:left-0}). Suppose that $\alpha\in\genset{D}$ is a left zero, let $T$
  be an arbitrary terminal component, and let $a,b\in T$.  By
  Lemma~\ref{lem:terminal1}, there exist $\beta,\gamma\in\genset{D}$ such that
  $t\beta=a$ and $t\gamma=b$ for all $t\in T$.  Let $t\in T$ be arbitrary.
  Since $T$ is terminal, $t\alpha\in T$.  But then, since $\alpha$ is a left
  zero,
  $a=(t\alpha)\beta=t(\alpha\beta)=t\alpha=t(\alpha\gamma)=(t\alpha)\gamma=b$.
  This shows that $|T|=1$.

  Conversely, suppose that all terminal components of $D$ are trivial:
  $\{t_1\}, \dots, \{t_k\}$. By Lemma~\ref{lem:terminal2}, there exists
  $\alpha\in\genset D$ such that $v \alpha$ belongs to a terminal component for
  each vertex $v\in\nset$. Now let $\beta\in\genset{D}$ be
  arbitrary.  Let $v\in\nset$ and put $v\alpha=t_j$.  Then
  $v\alpha\beta=t_j\beta=t_j=v\alpha$, so that $\alpha\beta=\alpha$, and
  $\alpha$ is a left zero.

  On the other hand, suppose that all of the terminal components of $D$ are
  trivial, and that $\alpha\in\genset{D}$ is such that $v \alpha$ does not
  belong to a terminal component for some $v\in\nset$.  Then $D$
  contains some arc $(v\alpha\to j)$.  But then $\alpha\not=\alpha(v\alpha\to
  j)$, so that $\alpha$ is not a left zero.

  (\ref{it:right-0}). Suppose that $D$ has at least two terminal components.
  Then there exist distinct terminal components $T_1$ and $T_2$ and a vertex
  $v$ such that there is a path from $v$ to a vertex from $T_1$ and a path from
  $v$ to a vertex from $T_2$. For $i=1,2$, let $\beta_i\in\genset{D}$ be such
  that $v \beta_i \in T_i$. If $\alpha \in \genset{D}$ is a right zero, then
  $$
        v \alpha = v \beta_1 \alpha \in T_1  \quad\text{and}\quad v \alpha = v
        \beta_2 \alpha \in T_2,
  $$
  which is the desired contradiction. 

  Conversely, suppose that $D$ has only one terminal component $T$ and fix some
  $t\in T$.  By Lemmas~\ref{lem:terminal1} and~\ref{lem:terminal2}, there exist
  $\alpha_1,\alpha_2\in\genset{D}$ such that $\im(\alpha_1)\subseteq T$, and
  $T\alpha_2=\{t\}$.  Then $\alpha_1\alpha_2\in\genset{D}$ is a constant map
  (with image $\{t\}$), and hence a right zero.

  On the other hand, suppose that $D$ only has one terminal component $T$, and
  that $\alpha\in\genset{D}$ is not a constant map.  So $\rk(\alpha)\geq2$.  We
  know that $\genset{D}$ contains some constant map $\beta$.  But then
  $\rk(\beta\alpha)=\rk(\beta)=1\not=\rk(\alpha)$, so that
  $\beta\alpha\not=\alpha$, whence $\alpha$ is not a right zero.
\end{proof}

A semigroup $S$ is called a \textit{left zero semigroup} if every element of
$S$ is a left zero; \textit{right zero semigroups} are defined analogously.

\begin{proposition}\label{prop-left-zeros}
  Let $D$ be a digraph. Then the following are equivalent:
  \begin{enumerate} [\rm(i)]

    \item
      $\genset{D}$ is a left zero semigroup;

    \item
      $\genset{D}$ has a unique $\LL$-class;

    \item
      there is a unique non-trivial connected component $K$ of $D$, and $K^{-1} = \set{(y\to x)}{(x\to y) \in K}$ is a fan.
  \end{enumerate}
\end{proposition}
\begin{proof}
  (i) $\Rightarrow$ (ii). Every left zero semigroup has a unique
  $\LL$-class.

  (ii) $\Rightarrow$ (iii). 
  Since $\genset{D}$ has a unique $\LL$-class, it follows that all of the arcs
  in $D$ belong to the same $\LL$-class. Hence the arcs in $D$ have the same
  image: say, $\nset\setminus \{x\}$ for some fixed $x\in\nset$. In other words, the arcs in $D$  are all of the form $(x\to
  y)$, $y\not = x$.  

  (iii) $\Rightarrow$ (i). 
  If the only arcs in $D$ are of the form $(x\to y)$ for some fixed $x$, then 
  $(x \to y)(x\to z) = (x \to y)$
  for all $y, z$, and so $\genset{D}=D$ is a left zero semigroup.
\end{proof}

\begin{proposition}\label{prop-right-zeros}
  Let $D$ be a digraph. Then the following are equivalent:
  \begin{enumerate}[\rm(i)]

    \item
      $\genset{D}$ is a right zero semigroup;

    \item
      $\genset{D}$ has a unique $\R$-class;

    \item
      there is a unique non-trivial connected component $K$ of $D$, and $K$ has
      $2$ vertices.
  \end{enumerate}
\end{proposition}
\begin{proof}
  (i) $\Rightarrow$ (ii). Every right zero semigroup has a unique
  $\R$-class.

  (ii) $\Rightarrow$ (iii). 
  Since $\genset{D}$ has a unique $\R$-class, all elements of $D$ are
  $\R$-related.  But $(i\to j)\R(k\to l)$ if and only if $\{i,j\}=\{k,l\}$
  and so all of the arcs of $D$ involve the same two vertices. 
  
  (iii) $\Rightarrow$ (i).
  In this case, $\genset{D}$ is isomorphic to a subsemigroup of $\Sing_2$,
  which is a right zero semigroup, and hence $\genset{D}$ is a right zero
  semigroup also.
\end{proof}

Recall that a semigroup is \textit{simple} if it contains a single $\J$-class.

\begin{proposition}
  Let $D$ be a connected digraph. Then the following are equivalent: 
  \begin{enumerate}[\rm(i)]

    \item
      $\genset{D}$ is a rectangular band;

    \item
      $\genset{D}$ is simple;

    \item
      $\genset{D}$ is a left or right zero semigroup.
  \end{enumerate}
\end{proposition}

\begin{proof}
(iii) $\Rightarrow$ (ii).  Every left or right zero semigroup is simple.

(ii) $\Rightarrow$ (i).  If $\genset D$ is simple, then it is completely
  regular.  If $n\geq3$, then we conclude from
  Proposition~\ref{prop:completely_regular} that $\genset D$ is a band; this is
  also true if $n\leq2$.  Every simple band is a rectangular band.

(i) $\Rightarrow$ (iii).  Suppose that $\genset D$ is a rectangular band.
  Since $\genset D$ has a single $\J$-class, every element of $\genset D$ has
  rank $n-1$. It follows that $\genset D$ consists entirely of arcs.  If
  $\genset D$ is not a left or right zero semigroup, then there exist distinct
  arcs $\alpha,\beta,\gamma\in\genset D$ such that $\alpha\R\beta$ and
  $\beta\LL\gamma$.  The former implies that $\alpha=(a\to b)$ and $\beta=(b\to
  a)$ for some $a,b$, and the latter that $\gamma=(b\to c)$ for some $c$, where
  $a,b,c\in\nset$ are distinct.  But then $\gamma\alpha\in\genset D$ is not an
  idempotent, a contradiction.
\end{proof}

Recall that a semigroup $S$ with a zero element $0$ is \textit{0-simple} if
$S^2\not=\{0\}$ and its $\J$-classes are $\{0\}$ and $S\setminus\{0\}$.

\begin{proposition}\label{prop-0-simple}
  Let $D$ be a digraph. Then $\genset{D}$ is
  $0$-simple if and only if the only non-trivial connected component of $D$ is
  one of the following:
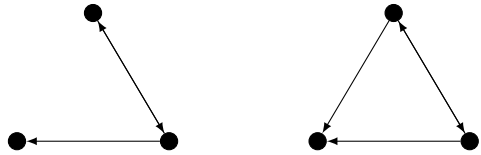
\vspace{\baselineskip}

  \centering
  \begin{tikzpicture}[vertex/.style={circle, 
    draw, 
    fill=black,
    inner sep=0.08cm}]
    \vertex{1}{0}{0} 
    \vertex{2}{1}{1} 
    \vertex{3}{2}{0} 

    \arc{2}{3}
    \arc{3}{2}
    \arc{3}{1}

  \end{tikzpicture} 
  \qquad\qquad
  \begin{tikzpicture}[vertex/.style={circle, 
    draw, 
    fill=black,
    inner sep=0.08cm}]
    \vertex{1}{0}{0} 
    \vertex{2}{1}{1} 
    \vertex{3}{2}{0} 

    \arc{2}{1}
    \arc{2}{3}
    \arc{3}{2}
    \arc{3}{1}

  \end{tikzpicture} 
\end{proposition}

\begin{proof}

  ($\Leftarrow$)
  It is straightforward to verify that $\genset{D}$ is 0-simple if $D$ is
  either of the given digraphs, using GAP~\cite{GAP4} for instance.

  ($\Rightarrow$) 
  Suppose that $\genset D$ is 0-simple and, without loss of generality, that
  $D$ has no isolated vertices.  Since $\genset D$ is 0-simple, it has two
  $\J$-classes, and the minimum one contains only the zero element.  In
  particular, $\genset D$ contains at most one element of rank smaller than $n
  -1$ and no elements of rank smaller than $n - 2$. 

  Suppose that $D$ has two connected components $D_1$ and $D_2$.  If there is
  an arc $\alpha$ in $D_1$, and distinct arcs $\beta,\gamma$ in $D_2$, then
  $\alpha\beta,\alpha\gamma\in\genset D$ are distinct elements of rank $n-2$, a
  contradiction. Hence, if $D$ has more than one connected component, then
  every connected component has exactly one arc. In this case, $\genset{D}$
  contains at least 2 arcs and the zero element and so $|\genset{D}| \geq 3$. 
  But Proposition~\ref{prop-j-trivial} implies that $\genset D$ is $\J$-trivial
  and so $\genset{D}$ has at least three $\J$-classes, a contradiction.  
  Thus $D$ is connected.  

  If $\alpha\in \genset{D}$ is the zero element, then by
  Proposition~\ref{prop:zero}, $\alpha$ is  constant which implies that $1 =
  \rk(\alpha) \geq n - 2$, and so $n \leq 3$. 
  It is possible to check that if $D'$ is any digraph with at most  3 vertices
  such that $\genset{D'}$ is 0-simple, then $D'$ is isomorphic to one of the
  two given digraphs. 
\end{proof}

\begin{proposition}\label{prop-congruence-free}
  Let $D$ be a digraph. Then $\genset{D}$ is
  congruence-free if and only if the only non-trivial connected component of $D$ is
  one of the following:
  
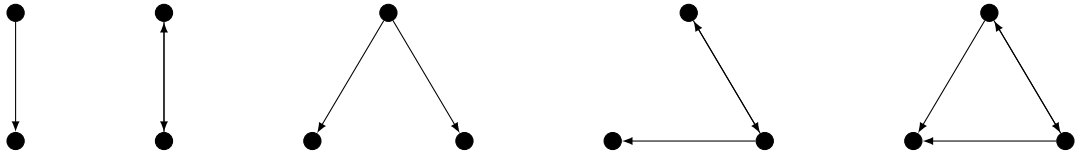
\vspace{\baselineskip}

  \centering
  \begin{tikzpicture}[vertex/.style={circle, 
    draw, 
    fill=black,
    inner sep=0.08cm} ]
    \vertex{1}{0}{1} 
    \vertex{2}{0}{0} 

    \arc{1}{2}

  \end{tikzpicture} 
  \qquad\qquad
  \begin{tikzpicture}[vertex/.style={circle, 
    draw, 
    fill=black,
    inner sep=0.08cm}]
    \vertex{1}{0}{1} 
    \vertex{2}{0}{0} 

    \arc{1}{2}
    \arc{2}{1}

  \end{tikzpicture} 
  \qquad\qquad
  \begin{tikzpicture}[vertex/.style={circle, 
    draw, 
    fill=black,
    inner sep=0.08cm}]
    \vertex{1}{0}{0} 
    \vertex{2}{1}{1} 
    \vertex{3}{2}{0} 

    \arc{2}{1}
    \arc{2}{3}

  \end{tikzpicture} 
  \qquad\qquad
  \begin{tikzpicture}[vertex/.style={circle, 
    draw, 
    fill=black,
    inner sep=0.08cm}]
    \vertex{1}{0}{0} 
    \vertex{2}{1}{1} 
    \vertex{3}{2}{0} 

    \arc{2}{3}
    \arc{3}{2}
    \arc{3}{1}

  \end{tikzpicture} 
  \qquad\qquad
  \begin{tikzpicture}[vertex/.style={circle, 
    draw, 
    fill=black,
    inner sep=0.08cm}]
    \vertex{1}{0}{0} 
    \vertex{2}{1}{1} 
    \vertex{3}{2}{0} 

    \arc{2}{1}
    \arc{2}{3}
    \arc{3}{2}
    \arc{3}{1}

  \end{tikzpicture} 
\end{proposition}
\begin{proof}
  Let $D_1, D_2, D_3, D_4, D_5$ (left to right) be the digraphs in the statement
  of the proposition. 

  ($\Leftarrow$) The semigroups $\genset{D_1},\genset{D_2},\genset{D_3}$ have size at most $2$, and so are congruence-free.  It is straightforward to verify that $\genset{D_4}$
  and $\genset{D_5}$ are both congruence-free (using GAP~\cite{GAP4} for
  instance).

  ($\Rightarrow$) If $\genset{D}$ is congruence-free, then either
  $|\genset{D}|\leq 2$, $\genset{D}$ is 0-simple, or $\genset{D}$ is a simple
  group; see \cite[Theorems 3.7.1 and 3.7.2]{H95}. So, by
  Propositions~\ref{prop-trivial} and~\ref{prop-0-simple}, it suffices to note
  that the only digraphs $D$ so that $|\genset{D}| = 2$ are the digraphs
  $D_2$ and $D_3$.
\end{proof}

\end{document}